\documentclass[11pt,twoside]{article}

\usepackage{xcolor}
\usepackage{amsmath}
\usepackage{amsfonts}
\usepackage{amsthm}
\usepackage{mathrsfs}
\usepackage{booktabs}
\usepackage{algorithm}
\usepackage{algpseudocode}
\usepackage{indentfirst}
\usepackage{bm}
\usepackage{bbm}
\usepackage{float}
\usepackage{mathtools}
\usepackage[normalem]{ulem}

\mathtoolsset{showonlyrefs}

\usepackage[margin=1in]{geometry}

\usepackage[round,authoryear]{natbib}

\usepackage[
    colorlinks=true,
    linkcolor=blue,
    citecolor=blue,
    urlcolor=blue
]{hyperref}

\newcommand{\field}[1]{\mathbb{#1}}
\newcommand{\R}{\field{R}}

\newcommand{\prob}{\field{P}}
\newcommand{\ind}{\mathbbm{1}}

\newcommand{\Kbar}{\overline{\mathbb{K}}}

\newcommand{\eps}{\varepsilon}
\newcommand{\Catom}[1]{C_{(#1})}

\newtheorem{problem}{Problem}
\newtheorem{assumption}{Assumption}
\newtheorem{theorem}{Theorem}
\newtheorem{lemma}{Lemma}
\newtheorem{definition}{Definition}
\newtheorem{remark}{Remark}
\newtheorem{proposition}{Proposition}
\newtheorem{corollary}{Corollary}

\begin{document}
\title{Retrieving predictive densities from conformal predictive distributions}

\author{
Dan Andrei Tudor$^{1}$, 
Dorota Kurowicka$^{2}$, 
Balint Negyesi$^{3}$, 
Valerii Zoller$^{3}$ \\[0.5em]
\small $^{1}$Department of Statistics, University of Wisconsin--Madison. \\
\small $^{2}$Delft Institute of Applied Mathematics, Delft University of Technology. \\
\small $^{3}$R\&D Labs, Ortec Finance, Rotterdam.
}

\date{}

\maketitle

\begin{abstract}
Conformal prediction intervals are used to construct marginally well-calibrated regions around point predictions. Recent work has focused on extending these ideas to conformal predictive distributions, whose marginal probability integral transform (PIT) follows a $\mathrm{Unif}[0,1]$ distribution. In this work, we investigate how to recover predictive densities from conformal predictive distributions. We extend the theory of conformal predictive distributions by proving asymptotic marginal validity for a tail-corrected version of conformal predictive distributions. We propose a method called \emph{quantile matching}, which preserves an upper bound on the deviation of the marginal PIT from uniformity. Furthermore, we show that the distribution induced by quantile matching is equivalent to the crisp version of conformal predictive distributions when the number of quantiles equals the size of the calibration set. For the recovery of conformal densities, we construct a fidelity-constrained bandwidth optimization for kernel smoothing that preserves asymptotic marginal validity and has a closed-form solution.
We apply and compare our proposed methodology on a large simulated real estate transactions dataset based on the Hierarchical Trend Model.

\end{abstract}


\section{Introduction}
Uncertainty quantification is a growing topic in statistics, and an important subfield is probability forecasting, i.e., instead of a single point prediction, constructing probability distributions over future values of a target variable conditional on the arrival of a number of features \citep{katzfuss}. Such a probabilistic forecast should be \emph{well-calibrated} and sharp, with calibration often being diagnosed through the probability integral transform (PIT), and its deviation from the uniform distribution -- see Lemma~\ref{randomizeduniform} in what follows.

\par The conformal framework of \cite{vovk} has been a turning point in model-agnostic, distribution-free probabilistic forecasting. Under the assumption of exchangeability alone, it covers any point predictor with a prediction set while maintaining a \emph{marginal coverage guarantee}. Its split version \citep{lei2017} reduces the computational cost to fitting a single point predictor and is the version used throughout this paper. For a recent survey on conformal prediction, we refer to \cite{angelopoulos_conformal_2023}. Much subsequent work has been concerned with making the resulting sets \emph{adaptive}, since the plain absolute-residual score yields intervals of fixed width. We mention conformal quantile regression \citep{CQR}, conformal direct prediction \citep{francke}, conditional histograms \citep{sesia_conformal_2021}, and density level sets \citep{lei_distribution-free_2013, izbicki_cd-split_2022}.  Other related conformal constructions that calibrate quantile- or distribution-based objects, but whose output is a prediction set rather than a predictive distribution, have been proposed by \citep{izbicki_flexible_2020}, \citep{gupta_nested_2022} and \citep{chernozhukov_distributional_2021}. These are the closest existing constructions
to the method proposed here, and we compare with them in
Section~\ref{sec:qm}. \color{black} All aforementioned works reshape the conformity score resulting in varying prediction set sizes.  Other strands improve sample efficiency \citep{barber_predictive_2021} or relax exchangeability \citep{barber_conformal_2023}. In all aforementioned works, and also in our paper, validity is meant marginally, as distribution-free \emph{conditional validity} on input features is known to be impossible without further structural assumptions, see \cite{foygel_barber_limits_2021}. 

A prediction interval reports where the target is likely to fall, but says nothing about the relative plausibility of the points inside it. Conformal predictive distributions (CPDs) close this gap \citep{lspm, split-cpd} as instead of fixing a single level $\alpha$, the conformal machinery is run at all levels simultaneously, and the output is a monotone, right-continuous object $Q_n$ whose probability integral transform is $\mathrm{Unif}[0,1]$ under exchangeability, thus modelling the \emph{cumulative distribution function} of the underlying uncertainty. The construction has since been extended to cross-conformal \citep{vovk_cross-conformal_2018}, kernelised \citep{vovk_conformal_2018} and Mondrian \citep{bostrom_mondrian_2021} variants, and used directly for decision making \citep{vovk_conformal_2018-1} and Bayesian computation \citep{fong_conformal_2021}.
A conformal predictive distribution has three critical properties in the context of our work. First, it is a piecewise constant step function supported on conformal atoms computed over a calibration set. Second, it uses a randomization parameter $\tau\sim \textnormal{Unif}[0, 1]$ that breaks ties in the calibration scores. Finally, its associated probability measure assigns non-zero mass to $\pm\infty$, meaning that CPDs are \emph{not} true distribution functions. These characteristics turn out to be a major obstacle when one tries to obtain an associated predictive density.

In high-risk settings such as finance, decisions do depend on the shape of uncertainty and with such full predictive distributions extreme-event probabilities and tail risk can be assessed. Works on conformal predictive distributions derive a cumulative distribution function of the predictive uncertainty that is marginally well-calibrated. However, in many applications a predictive \emph{density} may be desired, especially by practitioners. In fact, a density allows more direct qualitative insight into the shape of the distribution such as modality, skewness or tail behavior. Additionally, computing non-linear expectations of functionals of the target variable, logarithmic or entropy based scoring rules all require a density function. Most importantly, uncertainty in the target variable can give important insight in the presence of uncertain inputs. In the context of real-estate valuation, certain house characteristics determining the price might either be uncertain or outdated, such as energy label or state of maintenance. Denoting an uncertain feature by $\theta$ and all other input characteristics by $x$, a predictive density of the price conditional on all features $q(y\vert \theta, x)$ can be used in a Bayesian update $p(\theta\vert y, x)\propto q(y\vert \theta, x)p(\theta\vert x)$ to infer the distribution of the missing characteristics given a prior $p(\theta\vert x)$. Such probabilistic inference over latent or missing covariates necessitates predictive distributions.\color{black}

\par Motivated by the discussion above, in this article, we aim to extend conformal predictive distributions to predictive densities.  Naive finite differencing of a conformal predictive distribution results in a sum of Dirac masses, that may obscure certain features of the underlying density. Recovering a density from a CPD is therefore a true estimation problem. To counteract this issue, two main approaches are discussed. First, we propose a new method, \textit{quantile-matching}, in which we preselect a sequence of quantile levels at which we use the conformal algorithm to compute conformal quantiles, i.e. one-sided conformal prediction intervals. This allows the user to choose exactly the resolution they are interested in, which can be beneficial when practitioners are more interested in, say, more extreme quantiles. We show that quantile matching preserves a sharp theoretical upper bound on the PIT perturbation from the uniform distribution. Furthermore, we draw a connection between the quantile-matching induced distribution and the crisp modification \citep[Section~5]{split-cpd} of conformal predictive distributions, thereby obtain a novel marginal validity guarantee of the latter. Second, we use kernel smoothing on conformal distributions so obtained, derive and analyze a \emph{fidelity-constrained} bandwidth optimization problem that respects asymptotic marginal validity in the PIT distance. We prove that this optimization problem is well-posed, and has a closed-form solution in case of the Epanechnikov kernel.
We compare the models' performance on a large simulated dataset of real estate prices based on the Hierarchical Trend Model \citep{francke}, confirm the theoretical findings and demonstrate the advantages of (optimal kernel smoothed) quantile matching in terms of runtime and several scoring metrics, in particular, tail approximation errors.

\par Our main contributions are as follows.
\begin{itemize}
    \item The \emph{quantile-matching} method for recovering predictive densities from conformal predictive distributions, offering the user direct control over the resolution of the density estimate is proposed.
    \item A finite-sample bound on the PIT deviation for quantile-matched predictive distributions is proven (Theorem~\ref{qm-theorem}).
    \item A precise equivalence between quantile matching at full resolution and the crisp modification of conformal predictive distributions is established (Corollary~\ref{crispthm}).
    \item A finite-sample bound is proven for both tail-corrected conformal predictive distributions resulting in true distribution functions, and for their continuous approximations (Theorems~\ref{distance-tailcorrected}, \ref{thm:approximation_pit_guarantee}).
    \item A fidelity-constrained kernel smoothing optimization problem that selects the largest bandwidth while staying $\varepsilon$-close to the target conformal object. We show that this optimization problem is well-posed, preserves asymptotic validity and has a closed-form solution with the Epanechnikov kernel.
    \item We demonstrate empirically on a large simulated real estate dataset, for which the true conditional law is known, that the method robustly reduces density 'fuzziness' and is especially well-suited for evaluating tail statistics.
\end{itemize}

\par The paper is organized as follows. Section~2 provides an overview of the defining properties of empirical quantiles and distribution functions. Section~3 presents the construction of (split) conformal prediction intervals that serve as a basis to quantile matching. Section~4 introduces split conformal predictive systems. In Section~5 we discuss a measure correction that places the underlying conformal measure on the real line, which consequently induces a true distribution function. We prove asymptotic marginal validity of this correction and its continuous approximations. Section~6 introduces quantile matching and establishes its theoretical guarantees and its relation to the crisp modification of conformal predictive distributions. Section~7 develops the fidelity-constrained kernel-smoothing procedure. Section~8 presents the empirical comparison on simulated data.
\color{black}
\section{Empirical quantiles and distribution functions}

Let us first define empirical quantiles and give two relevant properties, variants of which can be consulted in \cite{CQR}. 
Denote the cumulative distribution function of a real-valued random variable $Z$ as $F(z):=~\prob(Z \leq z)$, and the associated \textit{true quantile function} at level $\alpha \in (0,1)$ as $q_\alpha:=\inf\{z \in \R: F(z) \geq ~\alpha\}$.
Similarly, the \textit{right-quantile function} is defined as $rq_\alpha:=\sup\{z \in \R: F^-(z)\leq \alpha\}$, where $F^-(z)=\prob(Z<z)$. The empirical cumulative distribution function and its associated empirical quantiles can be defined as follows.

\begin{definition}[Empirical quantiles]
\label{ecdf}
    In the case of identically distributed random variables $Z_1, \dots, Z_n$, we additionally define the \textit{empirical cumulative distribution function} as $\hat{F}_n(z):=\frac{1}{n}\sum_{i=1}^n \ind_{\{Z_i \leq z\}}$. We then define the \textit{empirical quantile function} $\hat{q}_{\alpha,n}$ as the true quantile function with respect to the empirical CDF.
\end{definition}

\begin{definition}[Right empirical quantiles]
    The right empirical quantile $\widehat{rq}_{\alpha,n}$  is the true right quantile function with respect to $\hat{F}^-_n(z):=\frac{1}{n}
\sum_{i=1}^n \ind_{Z_i<z}$.
\end{definition}

  The empirical and right-empirical quantile functions can be written down explicitly \citep{CQR} with respect to  the order statistics $Z_{(1)}, \dots, Z_{(n)}$ as 
\begin{equation}
    \label{empiricalquantile}
    \hat{q}_{\alpha,n}=Z_{(\lceil n\alpha\rceil)}, \quad \widehat{rq}_{\alpha,n} = Z_{(\lfloor n \alpha \rfloor+1)}.
\end{equation}

We now briefly summarize two lemmas concerning exchangeable random variables and their respective quantiles. Their proofs can be consulted in the supplementary material of \cite{CQR}. These two lemmas provide a bridge between the properties of quantiles and a theoretical marginal coverage guarantee typical of conformal procedures.
\begin{lemma}
\label{lemma1quantile}
    Let $Z_1, \dots, Z_n$ be exchangeable random variables. Then \[ \prob(Z_n \leq \hat{q}_{\alpha,n}) \geq \alpha. 
    \] If, in addition, $Z_1, \dots, Z_n$ are almost surely distinct, then \[
    \prob(Z_n \leq \hat{q}_{\alpha,n}) \leq \alpha + \frac{1}{n}.   
    \] 
\end{lemma}
The next lemma \citep{CQR} concerns a similar result when another random variable is added to the set. Note that the statement still refers to the empirical quantile of the first $n$ observations. This is crucial for proving the theoretical coverage guarantee of the split conformal prediction method in the upcoming section. 
\begin{lemma}
\label{lemma2quantiles}
     Let $Z_1, \dots, Z_{n+1}$ be exchangeable random variables. Then \[ \prob(Z_{n+1} \leq \hat{q}_{\alpha(1+\frac{1}{n}),n}) \geq \alpha. 
    \] If, in addition, $Z_1, \dots, Z_{n+1}$ are almost surely distinct, then \[
    \prob(Z_{n+1} \leq \hat{q}_{\alpha(1+\frac{1}{n}),n}) \leq \alpha + \frac{1}{n+1}.   
    \] 
\end{lemma}

Similarly, for the case of conformal predictive distributions, the following lemmas \cite[Section~2]{vovk} aid in the understanding of the theoretical marginal validity guarantee of these systems. Their proofs can be consulted in \cite[Section~6]{vovk}. The lemmas essentially highlight that being monotonically increasing and having a uniform distribution in probability integral transform are defining properties of the distribution function of real-valued random variables. Let us first assume that the distribution is everywhere continuous.

\begin{lemma}
    Let $F$ be a continuous distribution function on $\mathbb{R}$, $Y$ a random variable with distribution $F$, and $Q:\mathbb{R} \rightarrow \mathbb{R}$ a non-decreasing function. If $Q(Y) \sim \mathrm{Unif}([0,1])$, then $Q=F$.
    \label{definingproperty}
\end{lemma}

If the distribution function is not everywhere continuous, we can state a similar result, but we must account for randomization. Below, we recall the definition of lexicographic order on $\mathbb{R}\times[0,1]$.

\begin{definition}[Lexicographic order]
    The lexicographic order $(y, \tau) \leq (y', \tau ')$ on $\mathbb{R} \times [0,1]$ is defined to mean that $y<y'$ or both $y=y'$ and $\tau \leq \tau'$.
    \label{lexicographicorder}
\end{definition}
With this definition in mind, we can now adjust the above lemma for general distribution functions \cite[Lemma~2]{lspm}. 
\begin{lemma}
    \label{randomizeduniform}
    Let $\prob_F$ be a probability measure on $\mathbb{R}$ with distribution function $F$, and let $Y$ be a random variable distributed as $\prob_F$. Let $U$ be the induced probability measure of the uniform distribution on $[0,1]$ and $\tau _U\sim \mathrm{Unif}[0,1]$ independent of $Y$. Define $Q:\mathbb{R}\times [0,1] \rightarrow \mathbb{R}$ as a non-decreasing function with respect to the lexicographic order on $\mathbb{R} \times [0,1]$. Suppose the image $(\prob_F \times U)Q^{-1}$ of the product $\prob_F \times U$ under the mapping $Q$ is uniform on $[0,1]$, that is, $Q(Y,\tau_U) \sim \mathrm{Unif[0,1]}$. Then, for all $y$ and $\tau$,
    \begin{equation}
        \label{equalityindistribution}
        Q(y, \tau)=(1-\tau)F(y-)+\tau F(y).
    \end{equation}  
    Here, $F(y-)=\lim _{z \rightarrow y^-} F(z)$.
\end{lemma}

The two lemmas above show that monotonicity and the uniformity of PIT are defining properties of distribution functions. 
\par We can now move to present the standard construction of split conformal prediction intervals, based on \cite{vovk}.

\section{Conformal prediction intervals}

The Split Conformal Prediction Interval Method constructs a prediction interval that is marginally well-calibrated, regardless of sample size or the joint distribution of the samples. This method relies on loose \textit{exchangeability} assumptions of the observations, and the general conformal principle in the context of prediction intervals which was first introduced by Vovk in the early 2000s \citep{vovk}.

\par Under the \textit{exchangeability} assumptions, the method begins by splitting the observations into two disjoint subsets: a training set $\{(x_i, y_i): \, i \in \mathcal{I}_1\}$ and a calibration set $\{(x_i, y_i): i \in \mathcal{I}_2\}$, where $\mathcal{I}_1$ and $\mathcal{I}_2$ form a partition of $\{1, \dots, n\}$. For example, we can assume without loss of generality that $\mathcal{I}_1=\{1, \dots, m\}$ and $\mathcal{I}_2=\{m+1,\dots,n\}$. Given any regression algorithm $\mathcal{A}$, the training set is used for regression 
\[
\hat{y}(x) \leftarrow \mathcal{A}(\{(x_i, y_i) \,: \, i \in \mathcal{I}_1\}).
\]
Then, on $\mathcal{I}_2$, \textit{calibration scores} are computed between the target variable and the fitted regression model, in the form of absolute residuals:
\begin{equation}
    \label{cpscore}
E_i=|y_i-\hat{y}(x_i)|, \quad i \in \mathcal{I}_2.
\end{equation}
The empirical quantile of the absolute residuals is then computed, adjusted to the size of the calibration set, i.e. we redefine our adjusted confidence level $\tilde{\alpha}$ such that  $1 - \tilde{\alpha}=(1-\alpha)(1+\frac{1}{|\mathcal{I}_2|})$. The empirical quantile will then be $\hat{q}_{1-\tilde{\alpha},|\mathcal{I}_2|}(\mathcal{E})$ of $\mathcal{E}:=\{E_i : i \in \mathcal{I}_2\}$. The prediction interval for a new observation $x_{n+1}$ is then computed as
\begin{equation}
    \label{pi-conformalvolk}
    PI(x_{n+1})=[\hat{y}(x_{n+1})-\hat{q}_{1-\tilde{\alpha},|\mathcal{I}_2|}(\mathcal{E}), \, \hat{y}(x_{n+1})+\hat{q}_{1-\tilde{\alpha},|\mathcal{I}_2|}(\mathcal{E})].
\end{equation}
The procedure described here is also summarized in Algorithm~\ref{cp-method}.
\begin{algorithm}
\caption{Split Conformal Prediction Interval Method}
\label{cp-method}
\begin{algorithmic}[0]
\State \textbf{Input:} Dataset $\{(x_i, y_i)\}_{i=1}^n$, new observation $x_{n+1},$ miscoverage rate $\alpha \in (0,1)$, regression algorithm $\mathcal{A}$.
\State \textbf{Algorithm:} Partition $\{1, \dots, n\}$ into a training set $\mathcal{I}_1$ and a calibration set $\mathcal{I}_2$.
\State Fit  $\hat{y}(x) \leftarrow \mathcal{A} (\{(x_i, y_i)\}: i \in \mathcal{I}_1)$.
\State Compute conformity scores $E_i = |y_i - \hat{y}(x_i)|$ for $i \in \mathcal{I}_2$.
\State Compute the $1 - \tilde{\alpha} = (1-\alpha)(1+\frac{1}{|\mathcal{I}_2|})$ empirical quantile $\hat{q}_{1-\tilde{\alpha},|\mathcal{I}_2|}(\mathcal{E})$ of $\{E_i: i \in \mathcal{I}_2\}$.
\State \textbf{Output:}  
\[
\mathrm{PI}(x_{n+1}) = [\hat{y}(x_{n+1}) -\hat{q}_{1-\tilde{\alpha},|\mathcal{I}_2|}(\mathcal{E}),\; \hat{y}(x_{n+1})  + \hat{q}_{1-\tilde{\alpha},|\mathcal{I}_2|}(\mathcal{E})].
\]
\end{algorithmic}
\end{algorithm}

The split conformity score in \eqref{pi-conformalvolk} produces intervals of fixed width $2\hat{q}_{1-\tilde{\alpha}, |\mathcal{I}_2|}(\mathcal{E})$, the prediction being agnostic to varying uncertainty in different feature characteristics. In order to address this, several other conformity scores have been introduced in the literature. We mention \emph{conformalized quantile regression} \citep{CQR} and \emph{conformal direct prediction} \citep{francke}. In case of the latter, one first fits a quantile regression model on the residuals of the point prediction
\begin{align}
    \hat{q}^r_{1-\alpha}(x_i) \leftarrow \mathcal{B}(\{x_i, |y_i-\hat{y}(x_i)|: i\in \mathcal{I}_1\}),
\end{align}
then computes the conformity scores and its associated prediction interval as follows
\begin{align}
    \begin{aligned}[t]
        E_i&=\max\{\hat{y}(x_i) - \hat{q}^r_{1-\alpha}(x_i)-y_i, y_i-\hat{y}(x_i)-\hat{q}^r_{1-\alpha}(x_i)\}, \quad i\in\mathcal{I}_2,\\
    PI(x_{n+1})&=\left[\hat{y}(x_{n+1})-  \hat{q}^r_{1-\alpha}(x_{n+1}) - \hat{q}_{1-\tilde{\alpha}, |\mathcal{I}_2|}(\mathcal{E}), \hat{y}(x_{n+1})+ \hat{q}^r_{1-\alpha}(x_{n+1}) + \hat{q}_{1-\tilde{\alpha}, |\mathcal{I}_2|}(\mathcal{E})\right].
    \end{aligned}
    \label{eq:pi:direct}
\end{align}
\color{black}

For any of the aforementioned conformity scores, it is guaranteed that the prediction interval is marginally well-calibrated for exchangeable variables with this method, irrespective of sample size. Furthermore, an explicit upper bound can be given. For the split conformity score this was shown in \cite{lei2017}, and the same result is established for conformal quantile regression in \cite{CQR}. In what follows, we present the proof for conformal direct prediction.
\begin{theorem}
    \label{cpmarginalvalidity}
    Under the assumption that $(X_i, Y_i)_{i=1}^{n+1}$ are exchangeable, the prediction interval $PI(X_{n+1})$ constructed by the split conformal prediction, conformal quantile regression, or the conformal direct method is \textit{marginally} well-calibrated, i.e.,
    \[
    \prob(Y_{n+1} \in PI(X_{n+1})) \geq 1-\alpha.
    \]
    If, in addition, the conformity scores are almost surely distinct, then the prediction interval is nearly perfectly calibrated:
    \[
    \prob(Y_{n+1} \in PI(X_{n+1})) \leq 1-\alpha+\frac{1}{|\mathcal{I}_2|+1}.
    \]
\end{theorem}

    \begin{proof}
    For split conformal prediction, and conformal quantile regression see \cite{lei2017} and \cite{CQR}, respectively. We complete the proof for the conformal direct prediction.

    Observe that 
    \begin{align}
        &\begin{aligned}Y_{n+1}\geq \hat{y}(X_{n+1}) - \hat{q}^r_{1-\alpha}(X_{n+1}) - &\hat{q}_{1-\tilde{\alpha}, |\mathcal{I}_2|}(\mathcal{E})\\
            \iff &\hat{y}(X_{n+1}) - \hat{q}^r_{1-\alpha}(X_{n+1}) -Y_{n+1}\leq \hat{q}_{1-\tilde{\alpha}, |\mathcal{I}_2|}(\mathcal{E}),
        \end{aligned}\\
        &\begin{aligned}Y_{n+1}\leq \hat{y}(X_{n+1}) + \hat{q}^r_{1-\alpha}(X_{n+1}) + &\hat{q}_{1-\tilde{\alpha}, |\mathcal{I}_2|}(\mathcal{E})\\
            \iff &Y_{n+1} -\hat{y}(X_{n+1}) - \hat{q}^r_{1-\alpha}(X_{n+1}) \leq \hat{q}_{1-\tilde{\alpha}, |\mathcal{I}_2|}(\mathcal{E}).
        \end{aligned}
    \end{align}
    Consequently, using the definition of $E_{n+1}$ in \eqref{eq:pi:direct}, we obtain
    \begin{align}
        Y_{n+1}\in PI(X_{n+1})\iff E_{n+1}\leq \hat{q}_{1-\tilde{\alpha}, |\mathcal{I}_2|}(\mathcal{E}).
    \end{align}
    Since the original pairs $(X_i, Y_i)_{i=1}^{n+1}$ are exchangeable, so are the conformity scores $E_i$, as measurable functions of exchangeable random variables. Therefore, we can apply Lemma~\ref{lemma2quantiles} to get
    \begin{align}
         \mathbb{P}(Y_{n+1}\in PI(X_{n+1}))= \mathbb{P}(E_{n+1}\leq \hat{q}_{1-\tilde{\alpha}, |\mathcal{I}_2|}(\mathcal{E}))\geq 1- \alpha.
    \end{align}
    Under the additional assumption of the conformity scores being almost surely distinct, we also get the upper bound
    \begin{align}
        \mathbb{P}(Y_{n+1}\in PI(X_{n+1}))= \mathbb{P}(E_{n+1}\leq \hat{q}_{1-\tilde{\alpha}, |\mathcal{I}_2|}(\mathcal{E}))\leq 1-\alpha + \frac{1}{1+|\mathcal{I}_2|}.
    \end{align}
\end{proof}
Since prediction intervals do not specify which values are more likely within the given interval, extensions to predictive distributions have arisen in the literature \cite{lspm, split-cpd}. The following section details the split variant of conformal predictive systems, including the necessary and sufficient conditions imposed on the conformity measures.

\section{Conformal predictive systems}\label{sec:cpd}
In the following section we recall the key theoretical background on conformal predictive distributions that serves as a necessary basis for the proposed measure-correction solution and its theoretical result in Theorem~\ref{distance-tailcorrected}. The results below are due to Vovk and co-authors derived in a sequence of works \citep{lspm, split-cpd}. At a high level, they aim to construct an object based on conformity scores that acts similarly to a cumulative distribution function, while retaining a marginal validity guarantee equivalent to the ones in Theorem~\ref{cpmarginalvalidity}. However, as we shall see later, this object, called a \emph{conformal transducer}, includes randomization to break ties among conformity scores and to account for its own discontinuities, in light of Lemma~\ref{randomizeduniform}.

\par This section proceeds in two steps. We first recall the necessary definitions and conditions for a split conformal transducer to qualify as a randomized predictive system (Definition~\ref{rps}). We then present the algorithm for constructing the conformal predictive distribution and identify the technical obstacle, i.e. mass placed at $\pm\infty$ in the underlying probability measure that motivates the measure correction of Section~5 and, ultimately, the quantile-matching approach of Section~6.

\begin{definition}[Randomized predictive system]
\label{rps}\sloppy
 A function $Q: (\mathbb{R}^{p+1})^{n+1} \times [0,1] \to [0,1]$ is called a \textit{randomized predictive system} (RPS) if it satisfies:
\begin{enumerate}
    \item \textit{Monotonicity}: For any fixed training sequence $(z_1, \dots, z_n) \in (\mathbb{R}^{p+1})^n$ and test feature $x_{n+1} \in \mathbb{R}^p$, the function $ Q(z_1, \dots, z_n, (x_{n+1}, y), \tau)$ is non-decreasing both in $y$ for a fixed $\tau$, and in $\tau$ for fixed $y$.
    \item \textit{Boundary Stability:} For all $(z_1, \dots, z_n) \in (\mathbb{R}^{p+1})^n$ and $x_{n+1} \in \mathbb{R}^p$,
    \begin{equation*}
        \lim_{y \to -\infty} Q(z_1, \dots, z_n, (x_{n+1}, y), 0) = 0 \; \text{and} \; \lim_{y \to \infty} Q(z_1, \dots, z_n, (x_{n+1}, y), 1) = 1.
    \end{equation*}
    \item \textit{Validity:} For any exchangeable sequence of random variables $Z_1, \dots, Z_{n+1}$ taking values in $\mathbb{R}^{p+1}$ and any $\tau \sim \mathrm{Unif}[0,1]$ independent of the sequence, the random variable  $Q(Z_1, \dots, Z_n, Z_{n+1}, \tau)$
    follows the uniform distribution on $[0,1]$, i.e.,
    \begin{equation}
        \label{rpsvalidity}
    \prob( Q(Z_1, \dots, Z_n, Z_{n+1}, \tau)\le \alpha) = \alpha \quad \text{ for all } \alpha \in [0,1].
    \end{equation}
\end{enumerate}

\end{definition}

\begin{definition}[Randomized predictive distribution]
    A \textit{randomized predictive distribution} (RPD) is defined as the function 
    \begin{equation}
        Q_n:(y,\tau) \in \mathbb{R} \times [0,1] \mapsto Q(z_1, \dots , z_n, (x_{n+1},y), \tau),
    \end{equation}
    which is the output of the randomized predictive system $Q$ on a training sequence $z_1, \dots, z_n$ and a test feature $x_{n+1}$, coupled with a random draw $\tau \sim \mathrm{Unif}[0,1]$.
\end{definition}

To construct an RPS using the split conformal procedure, certain restrictions must be imposed on the underlying conformity scores. Here, we define split conformity measures and the associated split conformity scores in accordance with \cite{split-cpd} and briefly summarize the necessary and sufficient conditions on the conformity scores in order to construct an RPS.

\begin{definition}[Split conformity measure]\sloppy
    A family of measurable functions $E_m:$   $(\mathbb{R}^{p+1})^{m+1}$ $\rightarrow \mathbb{R}$ with $m=1,2, \dots$, is called a \textit{split conformity measure}.
\end{definition}

As before, we shall split the sequence $z_1, \dots, z_n$ into two: a training set $\{z_i : i \in \mathcal{I}_1\}$ and a calibration set $\{z_i: i \in \mathcal{I}_2\}$ where $\mathcal{I}_1$ and $\mathcal{I}_2$ form a partition of the index set $\{1, \dots, n\}$. For ease of presentation, we shall assume without loss of generality, $\mathcal{I}_1=\{1, \dots, m\}$ and $\mathcal{I}_2=\{m+1, \dots, n\}$. Suppose we fix $m$; from now on, we omit the subscript $m$ when referring to a split conformity measure corresponding to a split with $m$ training objects. The corresponding split conformity scores are defined below.

\begin{definition}[Split conformity score]
    For a split conformity measure $E:$ $(\mathbb{R}^{p+1})^{m+1}$ $\rightarrow~\mathbb{R}$ and a fixed $y \in \mathbb{R}$, the corresponding \textit{split conformity scores} are defined by
    \begin{equation}
    \label{confscore}
        \begin{split}
            E_i&=E(z_1, \dots, z_m, (x_{m+i},y_{m+i})), \quad i=1,\dots, n-m.\\
            E^y&=E(z_1, \dots,z_m, (x_{n+1}, y)).
        \end{split}
    \end{equation}
\end{definition}
\begin{remark}
    We can now formally define the split conformity measure used in Section~3. Under the notation used in this section, the standard Split Conformal Prediction method uses the split conformity measure $E_{CP}(z_1,\dots,z_m, (x_{n+1}, y))=|y-\hat{y}(x_{n+1})|$. The corresponding split conformity scores are then given by \eqref{cpscore}.
\end{remark}

\begin{definition}[Split conformal transducer]
    The \textit{split conformal transducer} determined by a conformity measure~$E$ and its corresponding conformity scores is defined as 
    \begin{equation}
    \label{splitconformaltransducer}
    \begin{split} 
        Q(z_1, \dots,& z_n, (x_{n+1},y), \tau):=\\ &\frac{|\{i=m+1, \dots, n : \, E_i < E^y \}| + \tau |\{i=m+1, \dots, n : E_i=E^y\}|+\tau}{n-m+1}.
    \end{split}
    \end{equation}
    Conversely, a function is called a \textit{split conformal transducer} if it is the split conformal transducer of some split conformity measure.
\end{definition}

\begin{definition}[Split conformal predictive system]
    A \textit{split conformal predictive system} is a function that is both a split conformal transducer and a randomized predictive system.
\end{definition}

\begin{definition}[Split conformal predictive distribution]
    For a split conformal predictive system $Q$, its randomized predictive distribution $Q_n$ is called a \textit{split conformal predictive distribution}.
\end{definition}

It is noteworthy to observe that the standard property of validity \eqref{rpsvalidity} adapted to split conformity measures is satisfied automatically \cite[Theorem~11.1]{vovk}. It is also stated as a known fact in \cite[Section~3]{split-cpd}, and is summarized in the following theorem.

\begin{theorem}
\label{validitythm}
    Let $Q:(\mathbb{R}^{p+1})^{n+1}\times[0,1] \rightarrow [0, 1]$ be a split conformal transducer associated with some conformity measure $E$. If $Z_1, \dots, Z_n, Z_{n+1}$ are exchangeable, where $Z_{n+1}=(X_{n+1}, Y_{n+1})$, and $\tau \sim \mathrm{Unif}[0,1]$ is~independent of $Z_1, \dots, Z_n, Z_{n+1}$, then $Q(Z_1, \dots, Z_n,$ $ Z_{n+1},\tau)$ follows the uniform distribution on $[0,1]$.
\end{theorem}

Since the standard property of validity is automatically satisfied, the question of whether a split conformal transducer is a Randomized Predictive System boils down to checking whether the monotonicity and boundary stability conditions are satisfied. Vovk et al.\ give in \cite{split-cpd} the necessary and sufficient conditions for this to happen, which we summarize here. The following definitions are as in \cite[Section~2.2]{lspm}.

\begin{definition}[Monotonic split conformity measure]
    A split conformity measure $E$ is called monotonic if $E(z_1, \dots, z_m, (x,y))$ is monotonically increasing in $y$, i.e.,
    \[
    y \leq y' \Rightarrow E(z_1, \dots, z_m, (x,y)) \leq E(z_1, \dots, z_m, (x,y')),
    \]
    for all $x$.
\end{definition}

\begin{definition}[Balanced split conformity measure]
    A monotonic split conformity measure $E$ is balanced if, for all $x$ and, for any $m$ and $z_1, \dots, z_m$, the set 
    \[
    \mathrm{conv} \, E(z_1, \dots, z_m, (x, \mathbb{R})):= \mathrm{conv}\{E(z_1,\dots, z_m, (x,y)) : y \in \mathbb{R}\}
    \]
    does not depend on $x$ and is an open interval in $\mathbb{R}$.
\end{definition}

\begin{remark}
    In the context of conformal prediction, an underlying assumption is that the conformity measure is \textit{coercive}, in the sense of $\mathrm{conv}\, E(z_1, \dots, z_m, \mathbb{R}^{p+1})= (-\infty, +\infty)$. Intuitively, this means that as $y$ drifts further away from our proposed model prediction, the (split) conformity measure notices this aspect and reflects it in its associated conformity score. We carry this assumption from now on. 
\end{remark}

\begin{theorem}
    The split conformal transducer based on a split conformity measure $E$ is an RPS if~and~only~if $E$ is balanced and monotonic. 
    \label{splittransducerrps}
\end{theorem}

\begin{corollary}
    The split conformal transducer based on the following modified split conformity measure, stemming from the standard conformal prediction method, is an RPS:
  \begin{equation}
E_{CP}(z_1, \dots, z_m, (x_{n+1},y))
= y - \hat{y}(x_{n+1}).
\label{modified-cp-score}
\end{equation}
\end{corollary}

Now that we have described the types of split conformity measures that produce a conformal predictive system, we give below the general algorithm to produce a conformal predictive distribution. Note that, by construction, the conformal predictive distribution still depends on $\tau \sim \mathrm{Unif}[0,1]$. As we let $\tau$ travel from $0$ to $1$, we obtain an interval for each fixed $y$. This is a ``fuzzy'' distribution and not a true distribution in the sense of producing a single output for a fixed $y$. We first give the pseudocode algorithm to create 'fuzzy' predictions \cite[Algorithm~1]{lspm}. We then adjust the construction to allow for randomized predictive distributions \cite[Algorithm~1]{split-cpd}.

\begin{algorithm}
\caption{'Fuzzy' Split Conformal Predictive System}
\label{fuzzy-method}
\begin{algorithmic}[0]
\State \textbf{Input:} Dataset $\{z_i=(x_i, y_i)\}_{i=1}^n$, observation (test object) $x_{n+1}.$ 
\State \textbf{Algorithm:} Partition $\{1, \dots, n\}$ into a training set $\mathcal{I}_1=\{1, \dots, m\}$ and a calibration set $\mathcal{I}_2=\{m+1, \dots, n\}$.
        \For{$i \in \{1, \dots, n-m\}$}
            \State Solve for $C_i$ in the equation $E(z_1, \dots, z_m, (x_{m+i}, y_{m+i}))=E(z_1, \dots, z_m, (x_{n+1}, C_i))$.
        \EndFor
\State Sort $C_1, \dots, C_{n-m}$ in ascending order to obtain $C_{(1)}\leq \dots \leq C_{(n-m)}$, and set $C_{(0)}=-\infty$ and $C_{(n-m+1)}=+\infty$.

\State \textbf{Output:}  Return a "fuzzy" predictive distribution for the label $y$ of $x_{n+1}$
\begin{equation}
    \label{fuzzypd}
    Q^{\text{fuzzy}}_n(y):=\begin{cases}
        \left[\frac{i}{n-m+1}, \frac{i+1}{n-m+1}\right] \quad \text{if } y \in (C_{(i)}, C_{(i+1)}) \text{ for } i \in \{0,1,\dots,n-m\}.\\
        \\
        \left[\frac{i'-1}{n-m+1}, \frac{i''+1}{n-m+1} \right] \quad \text{if } y=C_{(i)} \text{ for } i \in \{0,1,\dots,n-m\},
    \end{cases}
\end{equation}
    where $i':=\min\{j : C_{(j)}=C_{(i)}\}$ and $i'':=\max\{j:C_{(j)}=C_{(i)}\}$.
\end{algorithmic}
\end{algorithm}

The `fuzzy' split conformal predictive system can be easily modified to obtain a pseudo-distribution function by simply drawing a $\tau$ value and using it as an input. The procedure is described in Algorithm~\ref{cpd-algorithm}.

\begin{algorithm}
\caption{Split Conformal Predictive System}
\label{cpd-algorithm}
\begin{algorithmic}[0]
\State \textbf{Input:} Dataset $\{z_i=(x_i, y_i)\}_{i=1}^n$, observation (test object) $x_{n+1}$. 
\State \textbf{Algorithm:} Draw $\tau \sim \mathrm{Unif}[0,1]$.
\State Partition $\{1, \dots, n\}$ into a training set $\mathcal{I}_1=\{1, \dots, m\}$ and a calibration set $\mathcal{I}_2=\{m+1, \dots, n\}$.
        \For{$i \in \{1, \dots, n-m\}$}
            \State Solve for $C_i$ in the equation $E(z_1, \dots, z_m, (x_{m+i}, y_{m+i}))=E(z_1, \dots, z_m, (x_{n+1}, C_i))$.
        \EndFor
\State Sort $C_1, \dots, C_{n-m}$ in ascending order to obtain $C_{(1)}\leq \dots \leq C_{(n-m)}$, and set $C_{(0)}=-\infty$ and $C_{(n-m+1)}=+\infty$.
\State \textbf{Output:}  Return a predictive distribution for the label $y$ of $x_{n+1}$:
\begin{equation}
    \label{randomizedpd}
    Q^{CPD}_n(y, \tau):=\begin{cases}
        \frac{i+\tau}{n-m+1} \quad \text{if } y \in (C_{(i)}, C_{(i+1)}) \text{ for } i \in \{0,1,\dots,n-m\},\\
        \\
        \frac{i'-1+(i''-i'+2)\tau}{n-m+1}\quad \text{if } y=C_{(i)} \text{ for } i \in \{0,1,\dots,n-m\},
    \end{cases}
\end{equation}
    where $i':=\min\{j : C_{(j)}=C_{(i)}\}$ and $i'':=\max\{j:C_{(j)}=C_{(i)}\}$.
\end{algorithmic}
\end{algorithm}

In the two algorithms above, there are $n-m$ equations to be solved for $C_i$ (which we call the \textit{conformal atoms} from now on). This is not a computationally expensive step, as it often amounts to simply rearranging the score's output. Using the standard conformal prediction score~\eqref{modified-cp-score}, it amounts to solving
\[
C_i - \hat{y}(x_{n+1})=y_{m+i}-\hat{y}(x_{m+i}) \Leftrightarrow C_i= \hat{y}(x_{n+1})+y_{m+i}-\hat{y}(x_{m+i}).
\]

The conformal atoms are then merely a shifted version of the conformity scores $E_i$ associated with the conformity measure~\eqref{modified-cp-score}. In other words, we have the relation
\begin{equation}
    C_i=\hat{y}(x_{n+1})+E_i.
    \label{fromscorestoatoms}
\end{equation}

\par Furthermore, in view of \eqref{fromscorestoatoms}, under the common assumptions in Section~3 that the conformity scores are almost surely distinct, the conformal atoms are also almost surely distinct. The output of the conformal predictive system in Algorithm~\ref{cpd-algorithm} then simplifies to
\begin{equation}
    \label{randomizedpd-simplified}
    Q^{CPD}_n(y, \tau):=\begin{cases}
        \frac{i+\tau}{n-m+1} \quad \text{if } y \in (C_{(i)}, C_{(i+1)}) \text{ for } i \in \{0,1,\dots,n-m\},\\
        \\
        \frac{i-1+2\tau}{n-m+1}\quad \text{if } y=C_{(i)} \text{ for } i \in \{0,1,\dots,n-m\}.
    \end{cases}
\end{equation}

\par The split conformal predictive system with output $Q_n^{CPD}$ induces an underlying probability measure $\mu_n$ on $\Bar{\mathbb{R}}$ satisfying
\begin{equation}
    \mu_n(\cdot \mid x_{n+1}) := \frac{\tau}{n-m+1} \delta_{-\infty} + \sum_{i=1}^{n-m} \frac{1}{n-m+1} \delta_{C_{(i)}(x_{n+1})} + \frac{1-\tau}{n-m+1} \delta_{+\infty},
\label{cps-measure}
\end{equation}
where $\delta$ is the usual Dirac delta measure. We write $C_{(i)}(x_{n+1})$ to highlight the dependence on the observed test point. Then, the randomized predictive system $Q^{CPD}(\cdot, (x_{n+1},y), \tau)$, in view of~\eqref{splitconformaltransducer}, satisfies
\begin{equation}
    Q^{CPD}(\cdot, (x_{n+1},y),\tau)=\mu_n([-\infty,y)|x_{n+1})+\tau\mu_n(\{y\}|x_{n+1}),
    \label{randomized-standard-measure}
\end{equation}
where we use the shorthand notation $Q^{CPD}(\cdot,(x_{n+1},y),\tau)=Q^{CPD}(z_1,\dots,z_n,(x_{n+1},y),\tau)$.

We call the output $Q^{CPD}_n(y, \tau)$ a pseudo-distribution in Algorithm~\ref{cpd-algorithm} since for $y\in (-\infty, C_{(1)})$ the output is $\frac{\tau}{n-m+1}$, whereas for $y\in(C_{(n-m)}, +\infty)$ the output is $\frac{n-m+\tau}{n-m+1}$. In particular, for every $\tau \in [0,1]$, the conformal predictive system output does not satisfy $\lim_{y \rightarrow -\infty} Q^{CPD}_n(y, \tau)=0$ and $\lim_{y \rightarrow \infty} Q^{CPD}_n(y, \tau) =1$ simultaneously.

\par Cumulative distribution functions have underlying measures defined on $\mathbb{R}$ that give no weight at $\pm \infty$, so it is unclear how to retrieve a probability density function from a measure on the extended real line $\Bar{\mathbb{R}}$. This issue motivates the measure-correcting solution of Section~5, which reassigns the tail mass at $\pm \infty$ to $C_{(1)}$ and $C_{(n-m)}$ respectively, yielding an asymptotic randomized predictive distribution.

\begin{definition}[Asymptotic RPS]
    \label{asymptoticRPS}
    A function $Q:(\mathbb{R}^{p+1})^{n+1} \times [0,1] \rightarrow [0,1]$ is called an \textit{asymptotic randomized predictive system} if it satisfies the monotonicity and boundary stability conditions as in Definition~\ref{rps}, and the validity condition \eqref{rpsvalidity} holds asymptotically as $n \rightarrow \infty$: for any exchangeable sequence $Z_1, \dots, Z_{n+1}$ taking values in $\mathbb{R}^{p+1}$ and any $\tau \sim \mathrm{Unif}[0,1]$ independent of the sequence, the random variable  $Q(Z_1, \dots, Z_n, Z_{n+1}, \tau)$ converges in distribution to $\mathrm{Unif}[0,1]$.
\end{definition}

\section{A measure correction solution}\label{sec:measure_correction}

Denote by $\tilde{\mu}_n(\cdot \mid x_{n+1})$ the tail-corrected measure with formula 
\begin{equation}
\label{tailcorrectedmeasure}
\tilde{\mu}_n(\cdot \mid x_{n+1})
:=
\frac{1+\tau}{n-m+1}\,\delta_{C_{(1)}(x_{n+1})}+
\sum_{i=2}^{n-m-1}
\frac{1}{n-m+1}\,\delta_{C_{(i)}(x_{n+1})}
+
\frac{2-\tau}{n-m+1}\,\delta_{C_{(n-m)}(x_{n+1})}.
\end{equation}
This measure now lives on the real line $\mathbb{R}$ and admits a proper distribution function. Intuitively, the tail masses that were assigned to $\pm\infty$ in \eqref{cps-measure} are redistributed to the smallest and largest conformal atoms $C_{(1)}$ and $C_{(n-m)}$, respectively. The randomized predictive distribution associated with the tail-corrected measure is then given by
\begin{equation}
    \tilde{Q}^{CPD}(\cdot, (x_{n+1},y),\tau)=\tilde{\mu}_n((-\infty,y)|x_{n+1})+\tau\tilde{\mu}_n(\{y\}\vert x_{n+1}).
    \label{randomized-tailcorrected}
\end{equation}
We now show that applying this measure correction yields a predictive distribution that is approximately uniform and quantify the difference precisely. To do so, we use the concept of total variation between measures \cite{tvdistance}.

\begin{definition}
    The \textit{total variation distance} between two probability measures $\mu, \nu$ on the measurable space $(\Omega, \mathcal{F})$ is defined as
    \begin{equation}
        \mathrm{TV}(\mu, \nu)= \sup_{A \in \mathcal{F}} |\mu(A)-\nu(A)|.
        \label{standardtv}
    \end{equation}
\end{definition}

In our present context, the measures are purely atomic, so they are supported on a finite space. In this particular case, the total variation formula has a closed form, which we show in the following lemma. This lemma is a simple corollary to the Hahn--Jordan decomposition theorem \cite{hahnthm}.

\begin{lemma}
\label{lem:atomicTV}
Let $(\Omega,\mathcal F)$ be a measurable space and let
$\mu$ and $\nu$ be probability measures on $\Omega$.
Assume that there exists a finite set
$S=\{z_1,\dots,z_k\}\subset \Omega$ such that both measures are supported on $S$, i.e.\
$\mu(\Omega\setminus S)=0$ and $\nu(\Omega\setminus S)=0$.
Then
\begin{equation}
    \mathrm{TV}(\mu,\nu)
=
\frac12
\sum_{z\in S}
\left|
\mu(\{z\})-\nu(\{z\})
\right|.
\label{tv-discrete}
\end{equation}

\end{lemma}

\begin{theorem}
\label{distance-tailcorrected}
Let $Z_1, \dots, Z_n, Z_{n+1}=(X_{n+1}, Y_{n+1})$ be an exchangeable sequence, and $\tau\sim\textnormal{Unif}[0, 1]$ independent. Suppose a monotonic and balanced conformity measure and almost surely distinct atoms.
Let $\mu_n(\cdot \mid x_{n+1})$ be defined as in \eqref{cps-measure} and its associated tail-corrected measure $\tilde{\mu}_n(\cdot \mid~ x_{n+1})$ as in \eqref{tailcorrectedmeasure}. The associated randomized predictive distributions are constructed as in \eqref{randomized-standard-measure} and \eqref{randomized-tailcorrected}, respectively.
\par If $n-m \rightarrow \infty$ as $n \rightarrow \infty$, then  $\tilde{Q}^{CPD}(\cdot, (X_{n+1}, Y_{n+1}), \tau)$ is an asymptotic randomized predictive system. Furthermore, we have the bound
\[
\sup_{u\in[0,1]}
\left|
\prob(\tilde{Q}^{CPD}(\cdot, (X_{n+1}, Y_{n+1}), \tau)\le u) - u
\right|
\le
\frac{1}{n-m+1}.
\]
\end{theorem}

\begin{proof}
We first embed $\tilde{\mu}_n(\cdot\mid x_{n+1})$ into $\overline{\mathbb{R}}$
by assigning zero mass to $\pm\infty$, so both measures live on the
same measurable space. Since both measures are purely atomic, their total variation distance
is given by
\[
\mathrm{TV}(\mu_n(\cdot\mid x_{n+1}),\tilde{\mu}_n(\cdot\mid x_{n+1}))
=
\frac12
\sum_{z\in\overline{\mathbb{R}}}
\big|
\mu_n(\{z\}\mid x_{n+1})
-
\tilde{\mu}_n(\{z\}\mid x_{n+1})
\big|,
\]
in view of Lemma~\ref{lem:atomicTV}. The two measures differ only at four atoms: $-\infty$, $+\infty$, $C_{(1)}(x_{n+1})$, $C_{(n-m)}(x_{n+1}).$ A direct computation shows
\begin{align*}
   \big|
\mu_n(\{-\infty\}\mid x_{n+1})
-
\tilde{\mu}_n(\{-\infty\}\mid x_{n+1})
\big|&= \left|\frac{\tau}{n-m+1} - 0\right|
=
\frac{\tau}{n-m+1},\\
   \big|
\mu_n(\{C_{(1)}\}\mid x_{n+1})
-
\tilde{\mu}_n(\{C_{(1)}\}\mid x_{n+1})
\big|&=\left|
\frac{1}{n-m+1}
-
\frac{1+\tau}{n-m+1}
\right|
=
\frac{\tau}{n-m+1},\\
   \big|
\mu_n(\{C_{(n-m)}\}\mid x_{n+1})
-
\tilde{\mu}_n(\{C_{(n-m)}\}\mid x_{n+1})
\big|&=\left|
\frac{1}{n-m+1}
-
\frac{2-\tau}{n-m+1}
\right|
=
\frac{1-\tau}{n-m+1},\\
   \big|
\mu_n(\{+\infty\}\mid x_{n+1})
-
\tilde{\mu}_n(\{+\infty\}\mid x_{n+1})
\big|&=\left|\frac{1-\tau}{n-m+1} - 0\right|
=
\frac{1-\tau}{n-m+1}.
\end{align*}

Therefore,
\begin{equation*}
\mathrm{TV}(\mu_n(\cdot\mid x_{n+1}),\tilde{\mu}_n(\cdot\mid x_{n+1}))
=
\frac{1}{n-m+1}  \quad \text{for all }  x_{n+1} \in \mathbb{R}^{p}.
\end{equation*}

As a function of the random variable $X_{n+1}$, the above equality holds for every realization $x_{n+1}$ of $X_{n+1}$. Therefore, the equality also holds in the almost sure sense:

\begin{equation}
\label{eq:tvbound}
\mathrm{TV}(\mu_n(\cdot\mid X_{n+1}),\tilde{\mu}_n(\cdot\mid X_{n+1})) = \frac{1}{n-m+1} \quad \text{a.s.}
\end{equation}

Now consider the randomized predictive systems of~\eqref{randomized-standard-measure}~and~\eqref{randomized-tailcorrected}. For each $y$, define the function
\begin{equation}
    f_{y,\tau}(z)=\ind_{\{z<y\}}+\tau\,\ind_{\{z=y\}}.
\end{equation}
Then we can write
\begin{equation}
    Q^{CPD}(\cdot, (X_{n+1},y), \tau)=\int f_{y,\tau} \,d\mu_n, \quad \tilde{Q}^{CPD} (\cdot, (X_{n+1},y), \tau) =\int f_{y,\tau} d\tilde{\mu}_n.
\end{equation}

Furthermore, by construction we have $|f_{y,\tau}| \leq 1$ a.e., since it can only take values $0, \tau$, and $1$.  Then, using the layer cake representation formula and Jensen's inequality, we get
\begin{equation}
    \begin{split}
        | Q^{CPD}(\cdot, (X_{n+1}, y),\tau)-\tilde{Q}^{CPD}(\cdot,(X_{n+1},y),\tau)|
        &=\left|\int f_{y,\tau} \,d(\mu_n-\tilde{\mu}_n)\right|\\
        &=\left|\int_0^1 \mu_n(\{f>t\}) -\tilde{\mu}_n(\{f>t\}) \,dt\right|\\
        & \leq \int_0^1 \mathrm{TV}(\mu_n, \tilde{\mu}_n)\,dt\\
        &=\frac{1}{n-m+1}.
    \end{split}
    \label{inequalityinq}
\end{equation}
Let $u \in [0,1]$. Then, in view of the above inequality,
\begin{align*}
   &\{\tilde{Q}^{CPD}(\cdot, (X_{n+1}, Y_{n+1}), \tau) \leq u\}  \subseteq \left\{ Q^{CPD}(\cdot, (X_{n+1}, Y_{n+1}),\tau) \leq u+\frac{1}{n-m+1}\right\}, \\
    &\left\{Q^{CPD}(\cdot, (X_{n+1}, Y_{n+1}),\tau) \leq u-\frac{1}{n-m+1}\right\}  \subseteq \left\{ \tilde{Q}^{CPD}(\cdot, (X_{n+1}, Y_{n+1}), \tau) \leq u\right\}.
\end{align*}
Upon taking probabilities and using the fact that $Q^{CPD}(\cdot, (X_{n+1}, Y_{n+1}),\tau) \sim \mathrm{Unif}[0,1]$, we obtain
\[
\max \left\{0, u-\frac{1}{n-m+1} \right\} \leq \prob(\tilde{Q}^{CPD}(\cdot, (X_{n+1}, Y_{n+1}),\tau)\leq u) \leq \min \left\{1, u+\frac{1}{n-m+1} \right\}.
\]
This inequality holds for all $u \in [0,1]$. Therefore, we have the bound
\[
\sup_{u\in[0,1]}
\left|
\prob(\tilde{Q}^{CPD}(\cdot, (X_{n+1}, Y_{n+1}),\tau)\le u) - u
\right|
\le
\frac{1}{n-m+1},
\]
which immediately gives that $\tilde{Q}^{CPD}(\cdot, (X_{n+1}, Y_{n+1}), \tau)$ converges in distribution to $\mathrm{Unif}[0,1]$ as $n \rightarrow \infty$. The monotonicity of $\tilde{Q}^{CPD}$ follows from that of $Q^{CPD}$ together with the monotonicity of the redistribution of tail mass. Boundary stability holds by construction since $\tilde{\mu}_n$ places no mass at $\pm\infty$. Therefore, $\tilde{Q}^{CPD}$ is an asymptotic randomized predictive system in the sense of Definition~\ref{asymptoticRPS}.
\end{proof}
The randomized predictive distribution induced by the tail-corrected measure~\eqref{tailcorrectedmeasure} is \begin{equation} \label{tail-corrected-cpd-piecewise} \tilde{Q}^{CPD}_n(y,\tau) = \begin{cases} 0, & y<C_{(1)},\\[6pt] \dfrac{\tau(1+\tau)}{n-m+1}, & y=C_{(1)},\\[10pt] \dfrac{i+\tau}{n-m+1}, & y\in(C_{(i)},C_{(i+1)}), \quad i=1,\dots,n-m-1,\\[10pt] \dfrac{i-1+2\tau}{n-m+1}, & y=C_{(i)}, \quad i=2,\dots,n-m-1,\\[10pt] \dfrac{n-m-1+\tau+\tau(2-\tau)}{n-m+1}, & y=C_{(n-m)},\\[10pt] 1 & y>C_{(n-m)}. \end{cases} \end{equation}

Finally, one can extend the theoretical coverage guarantees of tail-corrected conformal predictive distributions to $\varepsilon$-close, continuous approximations thereof. The following result serves as the theoretical basis of the optimal filtering approach presented in Section~\ref{sec:optimal_filtering}. 

\begin{theorem}\label{thm:approximation_pit_guarantee}\sloppy
    Let the assumptions of Theorem~\ref{distance-tailcorrected} hold. Let $\widehat Q^{CPD}(\cdot, (X_{n+1}, Y_{n+1}), \tau)$ be a continuous approximation of $\tilde{Q}^{CPD}(\cdot, (X_{n+1}, Y_{n+1}), \tau)$ as defined in~\eqref{randomized-tailcorrected}. Suppose that for any $y\in \R$, we have
    \begin{align}\label{eq:epsilon_approximation}
        \left|\widehat Q^{CPD}(\cdot, (X_{n+1}, y), \tau) - \tilde{Q}^{CPD}(\cdot, (X_{n+1}, y), \tau)\right|\leq \varepsilon,
    \end{align}
    almost surely,
    for $\varepsilon>0$. Then the following bound holds
    \begin{align}\label{eq:approximation_pit_guarantee}
    \sup_{u\in[0,1]}
    \left|
    \prob(\widehat Q^{CPD}(\cdot, (X_{n+1}, Y_{n+1}), \tau)\le u) - u
    \right|
    \le
    \varepsilon + \frac{1}{n-m+1}.
    \end{align}
\end{theorem}

\begin{proof}
    Combining \eqref{inequalityinq} and \eqref{eq:epsilon_approximation} with the triangle inequality, we collect
    \begin{align}
        \left|\widehat Q^{CPD}(\cdot, (X_{n+1}, Y_{n+1}), \tau) - Q^{CPD}(\cdot, (X_{n+1}, Y_{n+1}), \tau)\right|\leq \begin{aligned}
            \varepsilon + \frac{1}{n-m+1}.
        \end{aligned}
    \end{align}
    Therefore, the following inclusions hold for any $u\in [0,1]$
    \begin{align*}
        &\left\{\widehat Q^{CPD}(\cdot, (X_{n+1}, Y_{n+1}), \tau)\leq u\right\}\subseteq \left\{ Q^{CPD}(\cdot, (X_{n+1}, Y_{n+1}), \tau)\leq u  + \varepsilon + \frac{1}{n-m+1}\right\},\\
        &\left\{Q^{CPD}(\cdot, (X_{n+1}, Y_{n+1}), \tau)\leq u  - \varepsilon - \frac{1}{n-m+1}\right\}\subseteq \left\{\widehat Q^{CPD}(\cdot, (X_{n+1}, Y_{n+1}), \tau)\leq u\right\}.
    \end{align*}
    Upon taking probabilities and, once again, using the fact that $Q^{CPD}(\cdot, (X_{n+1}, Y_{n+1}), \tau)\sim \textnormal{Unif}[0, 1]$, we obtain
    \begin{align*}
    \begin{aligned}[t]
        \max\left\{0, u-\varepsilon-\frac{1}{n-m+1}\right\}\leq \mathbb{P}(\widehat Q^{CPD}(\cdot, &(X_{n+1}, Y_{n+1}), \tau)\leq u)\\&\leq \min\left\{1, u+\varepsilon+\frac{1}{n-m+1}\right\}.
    \end{aligned}
    \end{align*}
    As this inequality holds for any $u\in [0,1]$ we recover the claim in \eqref{eq:approximation_pit_guarantee}.
\end{proof}
\section{Quantile matching: retrieving predictive densities}\label{sec:qm}
Conformal predictive distributions with their tail correction in Section~\ref{sec:measure_correction} produce a predictive distribution function supported on $\R$, with a bound on their deviation from the uniform distribution in the PIT. By their construction, CPDs are step functions supported on conformal atoms, and therefore, their derivative related to the associated density is a sum of Dirac masses and is thus ill-defined. The most straightforward approximation of the associated density is to use a forward finite difference scheme in between consecutive conformal atoms
\begin{align}
    \hat{f}(\cdot, (x,C_{(i)}))&=\frac{\tilde{Q}^{CPD}(\cdot, (x,C_{(i+1)}),\tau)-\tilde{Q}^{CPD}(\cdot,(x, C_{(i)}),\tau)}{C_{(i+1)}-C_{(i)}}=\frac{1}{(n-m+1)(C_{(i+1)}-C_{(i)})}.
    \label{direct-fd-pdf}
\end{align}
However, this simple approach leads to very noisy densities due to the large number of conformal atoms involved, and is very sensitive to small changes in between adjacent conformal atoms. 

In what follows, we propose a new approach called \emph{quantile matching} to decrease the amount of noise in the associated density approximations. This method merges the two worlds of conformal prediction intervals with conformal predictive distributions in a principled way: to reduce noise, a number of quantile levels is preselected and the associated conformal quantiles are computed. As we shall see later on, this essentially amounts to selecting a subset of the conformal atoms $C_{(i)}$. In order to recover predictive densities, we present two complementary approaches: finite differencing on this subset of conformal atoms leading to less noisy approximations (see \eqref{QM-pdf} below), and optimal kernel smoothing in Section~\ref{sec:optimal_filtering} acting on the piecewise constant predictive distribution function. With the new approach, we are able to both reduce the noise of the predictive density and maintain an upper bound on the perturbation in PIT of the associated predictive distribution.

The idea of building conformal estimates based on ordered quantiles has appeared in the literature before. Three recent works deserve an in-depth comparison. \cite{izbicki_flexible_2020} propose Dist-split, which calibrates the estimated PIT and returns a prediction interval based on it. \cite{gupta_nested_2022} replace the conformity score by a family of nested sets and calibrate them to find the minimal, valid prediction set. Finally, \cite{chernozhukov_distributional_2021} take a conditional distribution function as an input and compute an associated conformity score. The main differences between these approaches and quantile matching are as follows. First, quantile matching takes a single, scalar score as input instead of a conditional distribution \citep{chernozhukov_distributional_2021} or density \citep{izbicki_flexible_2020}. Second, the theoretical guarantee for quantile matching (see Theorem~\ref{qm-theorem} below) bounds the deviation from uniformity in PIT over finite samples for the whole \emph{distribution function}, whereas the aforementioned papers establish a coverage guarantee of a set for a single level. Finally, the main goal of quantile matching is the recovery of a conformal predictive density, which has not been addressed in the earlier works. The conditional density estimator in \cite{izbicki_flexible_2020}, FlexCode \citep{izbicki_converting_2017}, is a non-conformal input that shapes the prediction set without any associated theoretical guarantees.
\color{black}
In the previous sections, we have seen the notion of a randomized predictive system as the ''gold standard'' for constructing conformal predictive distributions. However, such constructions do not create a true conditional CDF, as their underlying measure lives on the extended real line $\Bar{\mathbb{R}}$ and gives weight at $\pm \infty$. Furthermore, they rely on an extra uniform random variable $\tau$, which introduces fuzziness and makes direct finite-differencing noisy.

\par We now propose a hybrid approach that removes the randomization entirely. We match a set of quantile levels with the output of one-sided conformal prediction intervals (Section~3) and recompute the predictive distribution for this subset of quantile levels (Section~4). This new framework, which we call the \textit{quantile-matching} method, directly constructs predictive distributions by computing conformal prediction quantiles. The main advantage of this method is that the user has additional freedom in choosing both the number and the locations of the quantile levels, which controls the resolution of the density estimate. Although we focus on evenly spaced quantile levels throughout this paper, the construction naturally extends to non-uniform grids, allowing additional
resolution in regions of practical interest such as extreme tails. Choosing a small number of quantile levels removes noisy patterns, while retaining a principled worst-case upper bound in terms of the distance from the true uniform distribution.

Let us first define predictive systems, following the framework presented in \cite[Section~2]{shen}. Predictive systems exclude randomization, removing one layer of complexity of randomized predictive systems while imposing similar regularity conditions and retaining a virtually unchanged validity guarantee. The advantage is that they construct predictive distributions on the real line. Excluding randomization makes them unsuitable for handling ties in the underlying data, but this does not act as an obstacle in our case under the standard assumption that the conformity scores are almost surely distinct.
\begin{definition}[Predictive system]
    \label{predictivesystem}
    A right-continuous function $Q:(\mathbb{R}^{p+1})^{n+1} \rightarrow [0,1]$ is called a \textit{predictive system} if it satisfies: 
    \begin{enumerate}
        \item \textit{Monotonicity:} For any fixed training sequence $(z_1, \dots, z_n) \in (\mathbb{R}^{p+1})^{n}$ and test feature $x_{n+1} \in \mathbb{R}^p$, the function $Q(z_1, \dots, z_n, (x_{n+1},y))$ is non-decreasing in $y$.
        \item \textit{Boundary stability:} For all $(z_1, \dots, z_n) \in (\mathbb{R}^{p+1})^n$ and $x_{n+1} \in \mathbb{R}^p$,
        \begin{equation*}
        \lim_{y \to -\infty} Q(z_1, \dots, z_n, (x_{n+1}, y)) = 0 \quad \text{and} \quad \lim_{y \to \infty} Q(z_1, \dots, z_n, (x_{n+1}, y)) = 1.
    \end{equation*}
    \item \textit{Validity:} For any exchangeable sequence of random variables $Z_1, \dots, Z_{n+1}$ taking values in $\mathbb{R}^{p+1}$, the random variable $Q(Z_1, \dots, Z_{n+1})$ follows the uniform distribution on $[0,1]$, i.e.,
    \begin{equation}
        \prob(Q(Z_1, \dots, Z_{n+1}) \leq \alpha)=\alpha \quad \text{for all } \alpha \in [0,1].
        \label{psvalidity}
    \end{equation}
    \end{enumerate}
    It is called an \textit{asymptotic predictive system} if \eqref{psvalidity} holds asymptotically as $n \rightarrow \infty$, i.e., the random variable $Q(Z_1, \dots, Z_{n+1})$ converges in distribution to the uniform distribution on $[0,1]$ as $n \rightarrow \infty$.
\end{definition}
\begin{definition}[Predictive distribution]
    The output of an (asymptotic) predictive system $Q$ is called an (asymptotic) \textit{predictive distribution} and is defined as the function
    \[
    Q_n:y \in \mathbb{R} \mapsto Q(z_1, \dots, z_n, (x_{n+1}, y)).
    \]
\end{definition}
\begin{remark}
    Observe that now, by construction, $Q_n$ is a true cumulative distribution function.
\end{remark}

We now describe an original construction of an asymptotic predictive system, starting from conformal prediction intervals. To align with the notation of Section~4, we assume again that $\mathcal{I}_2=\{m+1, \dots, n\}$, so that $|\mathcal{I}_2|=n-m$.

\par The main idea is to slightly modify the previously used conformity scores to obtain one-sided prediction intervals---approximate prediction quantiles---which can be recomputed over a fine grid of quantile levels. The following proposition formalizes this notion; it is a direct consequence of Theorem~\ref{cpmarginalvalidity}.

\begin{proposition}
    \label{propositionguarantee}
    For the standard one-sided conformity measure~\eqref{modified-cp-score}, the prediction interval created using its associated conformity scores is of the type $(-\infty, y_{X_{n+1}}^\alpha]$ and 
   \[
   \prob\left(Y_{n+1} \leq y_{X_{n+1}}^\alpha\right) \in \left[\alpha, \alpha+\frac{1}{n-m+1} \right].
   \]
\end{proposition}
\begin{remark}
    We can explicitly give the formula for $y_{X_{n+1}}^{\alpha}$ as follows:
\begin{align}
y_{X_{n+1}}^{\alpha}
&= \hat{y}(X_{n+1}) + \hat{q}_{\tilde{\alpha},|\mathcal{I}_2|}(\mathcal{E})=\hat{q}_{\tilde{\alpha},|\mathcal{I}_2|}(\mathcal{C}),
\label{cp-quantile}
\end{align}
    where the notation is aligned with Algorithm~\ref{cp-method}, $\mathcal{C}:=\left\{C_{(1)},...,C_{(n-m)}\right\}$ and $\tilde{\alpha}$ is the empirical adjusted quantile $\alpha \left(1+ 1/|\mathcal{I}_2|\right)$.
\end{remark}

Now suppose that we introduce an evenly-spaced grid $(\alpha_i)_{i=0}^K$ on $[0,1]$, hence $\alpha_i=i / K$ for $i=0,1,\dots,K$, with the convention that $y_{X_{n+1}}^{\alpha_0}=-\infty$, as well as $y_{X_{n+1}}^{\alpha_K}=+\infty$. Let us also define the empirical adjusted grid level with $\tilde{\alpha}_i:=\alpha_i \left(1+ 1/|\mathcal{I}_2|\right)$. Furthermore, to avoid overlap between predicted quantiles, we impose the strict condition $K<n-m+1$, in view of Proposition~\ref{propositionguarantee}. 

\par We can recompute the remaining predicted quantiles $(y_{X_{n+1}}^{\alpha_i})_{i=1}^{K-1}$. Then, for each interior grid level $i=1,\dots,K-1$, Proposition~\ref{propositionguarantee} gives \[ \prob\left( Y_{n+1}<y_{X_{n+1}}^{\alpha_i} \right) \in \left[ \alpha_i, \alpha_i+\frac{1}{n-m+1} \right]. \]
Plugging in an observed test point $X_{n+1}=x_{n+1}$, we
construct a right-continuous step function whose jump locations are the
predicted quantiles $y_{x_{n+1}}^{\alpha_i}$. Specifically, the function
takes the value $\alpha_i$ on
\[
\left[
y_{x_{n+1}}^{\alpha_{i-1}},
y_{x_{n+1}}^{\alpha_i}
\right),
\qquad i=1,\dots,K,
\]
which is equivalently expressed as
\begin{equation}
Q^{QM}(z_1,\dots,z_n,(x_{n+1},y))
=
\inf\left\{
\alpha_i:y<y_{x_{n+1}}^{\alpha_i}
\right\}.
\label{nontailcorrectedqm}
\end{equation}
Note, however, that in its current form the predictive system still does not output a true CDF: for $y<y_{x_{n+1}}^{\alpha_1}$, we have $Q^{QM}_n(y)=\alpha_1=1/K>0$. To correct this, we apply the following boundary correction at $C_{(1)}$, the smallest finite conformal atom, which by construction satisfies $y_{x_{n+1}}^{\alpha_1}=\hat{q}_{\tilde{\alpha}_1,|\mathcal{I}_2|}(\mathcal{C}) \geq C_{(1)}$:

\begin{align}
\tilde{Q}^{QM}(z_1, \dots, z_n, (x_{n+1}, y))
&=
\begin{cases}
\inf \{\alpha_i : y < y_{x_{n+1}}^{\alpha_i}\}
& \text{if } y \ge C_{(1)}, \\[4pt]
0
& \text{if } y < C_{(1)}.
\end{cases}
\label{quantileaps}
\end{align}

The procedure described above leading to \eqref{quantileaps} is called \textbf{quantile-matching}.

\begin{algorithm}
\caption{Split Quantile-Matched Predictive Distribution}
\label{randomized-algorithm}
\begin{algorithmic}[0]
\State \textbf{Input:} Dataset $\{z_i=(x_i, y_i)\}_{i=1}^n$, observation (test object) $x_{n+1}.$ 
\State \textbf{Algorithm:}  Partition $\{1, \dots, n\}$ into a training set $\mathcal{I}_1=\{1, \dots, m\}$ and a calibration set $\mathcal{I}_2=\{m+1, \dots, n\}$.
\State Introduce evenly-spaced grid $(\alpha_i)_{i=0}^K$ with $\alpha_i=i/K$ for all $i=0,1,\dots,K$
        \For{$i \in \{1, \dots,K-1\}$}
            \State Compute conformal prediction quantiles $y_{x_{n+1}}^{\alpha_i}$.
        \EndFor
\State \textbf{Output:}  Return a predictive distribution for the label $y$ of $x_{n+1}$
\begin{equation}
    \label{eq:qm-final}
   \tilde{Q}^{QM}_n(y):=\begin{cases}
\inf \{\alpha_i : y < y_{x_{n+1}}^{\alpha_i}\}
& \text{if } y \ge C_{(1)}, \\[4pt]
0
& \text{if } y < C_{(1)}.
\end{cases}
\end{equation}
\end{algorithmic}
\label{qm-algorithm}
\end{algorithm}

\par The following theorem gives an upper bound on the perturbation from the uniform distribution of the approximated system in~\eqref{quantileaps}. 
\begin{theorem}
    \label{qm-theorem}
    Suppose standard conformal prediction is used to construct the sequence of quantiles $(y_{X_{n+1}}^{\alpha_i})_{i=0}^K$. If $n-m \rightarrow \infty$ and $K \rightarrow \infty$ as $n \rightarrow \infty$, then the construction in $\eqref{quantileaps}$ is an asymptotic predictive system. Furthermore, we have the bound
    \begin{equation}
                \sup_{u \in (0,1)} \left|\prob(\tilde{Q}^{QM}(Z_1, \dots, Z_n, (X_{n+1}, Y_{n+1}))\leq u) -u\right| \leq \frac{1}{K}+\frac{1}{n-m+1}.
                \label{distanceqm}
    \end{equation}
\end{theorem}
\begin{proof}
    Note that, by construction, the sequence $(\alpha_i)_{i=0}^K$ is increasing, and so is $( y_{x_{n+1}}^{\alpha_i})_{i=0}^K$. This follows directly from the formula $y_{x_{n+1}}^{\alpha_i}=\hat{y}(x_{n+1})+\hat{q}_{\tilde{\alpha}_i, |\mathcal{I}_2|}(\mathcal{E})$ and the fact that the empirical quantile $\hat{q}_{\tilde{\alpha}_i, |\mathcal{I}_2|}(\mathcal{E})$ is non-decreasing in $\alpha$.
    Then, $\tilde{Q}^{QM}$ is monotonic in $y$, as an increase in $y$ can only reach a higher threshold level $y_{x_{n+1}}^{\alpha_i}$. Throughout this proof, we write $\tilde{Q}^{QM}(\cdot, (X_{n+1}, Y_{n+1})):=\tilde{Q}^{QM}(Z_1,\dots,Z_n, (X_{n+1}, Y_{n+1}))$.
    \par The boundary stability condition is satisfied by construction: for $y < C_{(1)}$, we have $$\tilde{Q}^{QM}(z_1,\dots,z_n, (x_{n+1},y))=0,$$ and for $y\ge y_{x_{n+1}}^{\alpha_{K-1}}$, we have 
    $\tilde{Q}^{QM}(z_1,\dots,z_n, (x_{n+1},y))=1$.
\par It remains to check the asymptotic validity property. First, note that, for each $i=1, \dots, K$, we have
\begin{equation}
\tilde{Q}^{QM}(z_1, \dots, z_n, (x_{n+1}, y)) \le \alpha_i \text{ if and only if } y<y_{x_{n+1}}^{\alpha_i}.
\label{ifandonlyif}
\end{equation}
Indeed, if $y< y_{x_{n+1}}^{\alpha_i}$, then $\inf \{\alpha_j : y < y_{x_{n+1}}^{\alpha_j}\} \le \alpha_i$, and the left-hand side is exactly the value that $\tilde{Q}^{QM}$ takes on this branch. Vice versa, if $\tilde{Q}^{QM}(\cdot, (x_{n+1},y)) \le \alpha_i$, then $\inf \{\alpha_j : y < y_{x_{n+1}}^{\alpha_j}\} \le \alpha_i$, so $y < y_{x_{n+1}}^{\alpha_i}$---otherwise, $y\ge y_{x_{n+1}}^{\alpha_i}$ would imply that $\inf \{\alpha_j : y < y_{x_{n+1}}^{\alpha_j}\} > \alpha_i$.
Then, using Proposition~\ref{propositionguarantee},
\begin{equation}  
\prob(\tilde{Q}^{QM}(Z_1, \dots, Z_n, (X_{n+1}, Y_{n+1})) \le \alpha_i) =\prob(Y_{n+1}<y^{\alpha_i}_{x_{n+1}}) \in \left[\alpha_i,\alpha_i+\frac{1}{n-m+1}\right],
\label{upperobound1}
\end{equation}
where the central equalities follow from the fact that the conformal atoms are almost surely distinct and Proposition~\ref{propositionguarantee}.
Therefore, the validity requirement holds at $u \in \{\alpha_1, \dots, \alpha_K\}$. In fact, the statement is even stronger:
\begin{equation}
     \left|\prob(\tilde{Q}^{QM}(Z_1, \dots, Z_n, (X_{n+1}, Y_{n+1}))\leq \alpha_i) -\alpha_i\right| \leq  \frac{1}{n-m+1} \quad \text{ for all } i=1,\dots,K.
\end{equation}

\par Now fix $u$ such that $\alpha_i<u<\alpha_{i+1}$ for some $i=1, \dots, K-1$. Then,
\begin{align}
\begin{aligned}
    \prob(\tilde{Q}^{QM}(\cdot, (X_{n+1}, Y_{n+1})) \leq \alpha_i) \leq \prob(\tilde{Q}^{QM}(\cdot, &(X_{n+1}, Y_{n+1})) \leq u)\\ &\leq \prob(\tilde{Q}^{QM}(\cdot, (X_{n+1}, Y_{n+1})) \leq \alpha_{i+1}).
\end{aligned}
\end{align}
Using \eqref{upperobound1}, we obtain the bounds
\[
\alpha_i \leq \prob(\tilde{Q}^{QM}(\cdot, (X_{n+1}, Y_{n+1})) \leq u) \leq \alpha_{i+1}+\frac{1}{n-m+1}.
\]
Subtracting $u$ from both sides and using $\alpha_{i}<u<\alpha_{i+1}$:
\[
\alpha_i -\alpha_{i+1} \leq \prob(\tilde{Q}^{QM}(\cdot, (X_{n+1}, Y_{n+1})) \leq u) -u \leq \alpha_{i+1}+\frac{1}{n-m+1}-\alpha_i.
\]
Taking absolute values:
\begin{equation}
\begin{aligned}
    |\prob(\tilde{Q}^{QM}(Z_1, \dots, Z_n, (X_{n+1}, Y_{n+1})) \leq u)-u| &\leq \alpha_{i+1}-\alpha_i+\frac{1}{n-m+1}\\&=\frac{1}{K}+\frac{1}{n-m+1}.
\end{aligned}
    \label{upperbound2}
\end{equation}
Finally, fix $u$ such that $0<u<\alpha_1$. Then,
\[
0 \leq \prob(\tilde{Q}^{QM}(\cdot, (X_{n+1}, Y_{n+1})) \leq u) \leq \alpha_1,
\]
so that
\begin{equation}
    |\prob(\tilde{Q}^{QM}(\cdot, (X_{n+1}, Y_{n+1})) \leq u)-u| \leq \max \{\alpha_1-u, u\} \leq \alpha_1=\frac{1}{K}.
    \label{upperbound3}
\end{equation}
Combining \eqref{upperobound1}, \eqref{upperbound2} and \eqref{upperbound3} gives the bound 
\[
    \sup_u \left|\prob(\tilde{Q}^{QM}(Z_1, \dots, Z_n, (X_{n+1}, Y_{n+1}))\leq u) -u\right|\leq\frac{1}{K}+\frac{1}{n-m+1}.
\]
Letting $n \rightarrow \infty$ and then $K \rightarrow \infty$ gives convergence to zero, hence asymptotic validity. Since all three conditions of Definition~\ref{predictivesystem} are satisfied, $\tilde{Q}^{QM}$ is an asymptotic predictive system.
\end{proof}
\begin{remark}[Using alternative conformity scores] \label{remark:cdp} The theoretical guarantees of Theorem~\ref{qm-theorem} rely on the monotonicity of the estimated conformal quantiles $(y_{x_{n+1}}^{\alpha_i})_{i=1}^{K-1}$ with respect to the quantile level $\alpha_i$. For the standard conformity score, this property follows automatically from the monotonicity of empirical quantiles. However, when using conformity scores such as conformal direct prediction (CDP), the estimated residual quantiles are obtained from a quantile regression model and are not theoretically guaranteed to be monotone across levels $\alpha_i$. Consequently, the quantile-matching construction may fail to define a valid predictive system and the coverage guarantee of Theorem~\ref{qm-theorem} no longer applies. In practice, this issue is often negligible for CDP. Since the quantile regression is performed on residual magnitudes rather than directly on the target variable, violations of monotonicity are typically small and were not observed in our experiments. Whenever the estimated quantiles remain monotone, the quantile-matching construction can be applied without modification and can subsequently be smoothed using the methods of Section~\ref{sec:optimal_filtering}. Monotonicity can also be enforced directly during model fitting. For example, when using LightGBM, one may train a joint model on $(x,\alpha)$ while imposing monotonicity in the quantile-level variable $\alpha$, thereby ensuring that estimated quantiles remain ordered. A practical advantage of CDP is that it produces genuinely heteroscedastic predictive distributions. Unlike the standard conformity score, which effectively translates a fixed residual distribution to the point prediction, CDP allows the shape and spread of the predictive distribution to vary with the covariates. Empirically, we found this effect to be particularly beneficial for smaller calibration sets, where CDP-based quantile matching produced more realistic predictive distributions and improved tail estimates. \end{remark}

\subsection*{Relation with conformal predictive distributions}

\par First, let us introduce the \textit{crisp} modifications of randomized predictive distributions \cite[Section~5]{split-cpd}:
\begin{equation}
    Q^{\text{crisp}}(z_1, \dots, z_n, (x_{n+1},y)):= \frac{i}{n-m} \quad \text{ if } y \in [C_{(i)}, C_{(i+1)}).
    \label{crisp}
\end{equation}
The crisp modification no longer depends on $\tau$ and essentially acts as the \textit{empirical CDF} of the conformal atoms $C_{(i)}$.

\par Let us now recall the connection between conformal quantiles and conformal atoms, highlighted in ~\eqref{cp-quantile}, 
\begin{equation}
    y_{x_{n+1}}^{\alpha_i}=\hat{q}_{\tilde{\alpha}_i,|\mathcal{I}_2|}(\mathcal{C}).
    \label{cpquantilefromatoms}
\end{equation}

That is, with quantile matching we preselect a group of quantile levels and subsequently compute the empirical quantiles of the conformal atoms at those levels. By~\eqref{empiricalquantile}, these have the closed form $\hat{q}_{\tilde{\alpha}_i,|\mathcal{I}_2|}(\mathcal{C})=C_{(\lceil(n-m)\alpha_i+\alpha_{i} \rceil)}$. The procedure therefore amounts to preselecting a subset of the conformal atoms to be matched with each quantile level.

\par When $K=N:=n-m$, the grid is $\alpha_i=i/N$. Since \[ \tilde{\alpha}_i = \alpha_i\left(1+\frac{1}{N}\right), \] the matched quantiles satisfy \[ y_{x_{n+1}}^{\alpha_i} = C_{\left(\left\lceil (N+1)\alpha_i\right\rceil\right)} = C_{(i+1)}, \qquad i=1,\dots,N-1. \] Consequently, for \[ y\in[C_{(i)},C_{(i+1)}), \qquad i=1,\dots,N-1, \] the smallest grid level whose matched quantile is strictly greater than $y$ is $\alpha_i=i/N$. Therefore, \[ \tilde Q_n^{QM}(y)=\frac{i}{N} = Q^{\mathrm{crisp}}_n(y). \] Moreover, both constructions equal zero for $y<C_{(1)}$ and one for $y\ge C_{(N)}$. Hence, at full resolution, the quantile-matched predictive distribution coincides exactly with the right-continuous crisp predictive distribution. This establishes the following corollary.

\begin{corollary}
    \label{crispthm}
    Let $Q^{\text{crisp}}$ be the crisp modification of randomized predictive distributions as defined in~\eqref{crisp}. Then the following upper bound holds:
    \begin{equation}
                \sup_{u \in (0,1)} \left|\prob(Q^{\text{crisp}}(Z_1, \dots, Z_n, (X_{n+1}, Y_{n+1}))\leq u) -u\right| \leq \frac{1}{n-m}+\frac{1}{n-m+1}.
                \label{distancecrisp}
    \end{equation}
\end{corollary}
\begin{proof}
    Apply Theorem~\ref{qm-theorem} with $K=n-m$, noting that the quantile-matching construction at this resolution coincides with~\eqref{crisp} as argued above.
\end{proof}

\par In the context of conformal prediction, the approach of quantile-matching is particularly advantageous. Usually, constructing a predictive distribution starting from conditional quantiles is a difficult task, as it involves the risk of quantile crossing, i.e. the estimated conditional quantiles are not increasing in $\alpha$. Imposing monotonicity involves major additional effort \cite{crossing}. However, in our case, given the straightforward relation \eqref{cpquantilefromatoms}, monotonicity of the estimated quantiles is automatically satisfied, in view of the fact that the empirical quantile function is automatically nondecreasing in $\alpha$.

\subsection*{Quantile-matched predictive density}

Given that the user has freedom in choosing both the number and the levels of the quantiles, the trade-off between a larger deviation from the uniform distribution and greater flexibility becomes more manageable. For instance, certain practitioners might be more interested in more extreme scenarios, and one could use a non-uniform grid concentrating mass where precision is most needed.

\par For the quantile-matching method, one can use a forward finite difference scheme to approximate the density at a given quantile level:
\begin{equation}
    \hat{f}_{QM}(\cdot, (x,y_x^{\alpha_i}))=\frac{\tilde{Q}^{QM}(\cdot, (x,y_x^{\alpha_{i+1}}))-\tilde{Q}^{QM}(\cdot,(x, y_{x}^{\alpha_{i}}))}{y_x^{\alpha_{{i+1}}}-y_x^{\alpha_i}}=\frac{\alpha_{i+1}-\alpha_{i}}{y_x^{\alpha_{{i+1}}}-y_x^{\alpha_i}},
    \label{QM-pdf}
\end{equation}
which, under an evenly spaced grid, translates to $\hat{f}_{QM}(\cdot, (x,y_x^{\alpha_i}))=1/(K(y_x^{\alpha_{{i+1}}}-y_x^{\alpha_i}))$.

Using fewer quantiles removes the noisy patterns at the density level, as we shall shortly see in our examples, while preserving the normalization of the density function. The user is responsible for picking the right number of quantiles, depending on their tolerance for the perturbation of the probability integral transform. The computational cost of this step is $O(K)$ in the density evaluation phase, compared to $O(n-m)$ for direct finite differencing, which yields a substantial speedup for large calibration sets compared to conformal predictive distributions. Nonetheless, \eqref{quantileaps} is still a step function and hence not differentiable, which motivates the optimal kernel smoothing method introduced below. The two approaches complement each other: quantile matching sets the resolution of the distribution to a user-defined set of quantile levels, while optimal smoothing ensures a differentiable, closed-form distribution that remains $\varepsilon$-close to its conformal target and preserves asymptotic validity.

\color{black}
\section{Optimal kernel smoothing of conformal distributions}\label{sec:optimal_filtering}

In the context of retrieving predictive densities, applying kernel smoothing on an empirical \citep{nadaraya_new_1964, azzalini_note_1981, bowman_bandwidth_1998} or tail-corrected conformal distribution function yields smooth, closed-form predictive densities that are analytically differentiable. However, this comes at the price of losing the marginal coverage guarantees established by Theorem \ref{distance-tailcorrected}. In what follows, we construct a constrained optimization problem that gives the highest kernel bandwidth for which the smoothed CDF remains $\varepsilon$-close to its target conformal object. 
The difference between our approach and previous techniques in the literature is two-fold. First, unlike earlier works, see e.g. \cite{swanepoel_mean_1988, jones_performance_1990}, our optimization problem is not aimed to minimize a distance in distribution such as the mean integrated square error, but rather maximizes the bandwidth. Second, we impose distance from the conformal distribution target as a \emph{constraint}.
Therefore, by virtue of Theorem~\ref{thm:approximation_pit_guarantee}, we preserve an asymptotic coverage guarantee while directly controlling the associated PIT estimate in finite samples. Finally, we show that in case of the Epanechnikov kernel, this optimization problem has a closed-form solution.

The common key characteristic of the two conformal CDFs above is that they are \emph{piecewise-constant}, right-continuous functions fully characterized by the locations and sizes of jumps between two adjacent atoms.
To smooth a general (possibly tail-corrected CPD \eqref{randomized-tailcorrected} or quantile-matched \eqref{quantileaps}) conformal
object, we abstract $\tilde{Q}(\cdot, (X_{n+1}, Y_{n+1}))$ by its jumps.
\begin{definition}[Atoms, levels, jump masses]
\label{def:jumps}
Let $Q_n$ have jumps at ordered locations $\Catom{1} < \cdots < \Catom{M}$ with
right-limit levels $F_i := Q_n(\Catom{i}^{+})$ and $F_0 := Q_n(\Catom{1}^{-})$.
Define the \emph{jump masses}
\begin{equation}\label{eq:jump_mass}
\Delta_i := F_i - F_{i-1} \ge 0, \qquad i = 1,\dots,M,
\end{equation}
so that the step CDF is 
\begin{align}\label{eq:heaviside_cdf}
	Q_n(y) = \sum_{i=1}^M \Delta_i\, H(y - \Catom{i})
\end{align}
with
$H$ the Heaviside step. For the tail-corrected CPD \eqref{tail-corrected-cpd-piecewise},
$M = N$ and $\Delta_1=(1+\tau)/(N+1), \Delta_i \equiv 1/(N+1), i=2, \dots, N-1$, $\Delta_N=(2-\tau)/(N+1)$; for the quantile-matching discretizations
with $K$ quantile levels, $M = K$ and the $\Delta_i$ are the level increments $\alpha_i - \alpha_{i-1}, i=1, \dots, K$.
\end{definition}

The goal is to smooth the piecewise-constant CDFs given either by the tail-corrected CPD of \eqref{tail-corrected-cpd-piecewise} or by the quantile-matched predictive distribution in \eqref{eq:qm-final}, in a way that the associated PDFs are less noisy than in the naive finite differenced case. We propose to achieve this by kernel smoothing. Recall, see e.g. \cite{silverman_density_2018}, that a kernel $K$ is a symmetric,
non-negative, integrable function with $\int K = 1$; for bandwidth $h>0$ we
write $K_h(u) = h^{-1}K(u/h)$ and $\mathbb{K}(t) \coloneqq \int_{-\infty}^{t} K(s)\,ds$
for the (unit-bandwidth) kernel CDF, with $\mathbb{K}_h(z) = \mathbb{K}(z/h)$.

Smoothing means convolving the step CDF with the kernel,
$Q_n^{h} := Q_n * K_h$. Because convolution is linear and a single step
convolved with a kernel is a shifted kernel CDF,
$H(\cdot - \Catom{i}) * K_h = \mathbb{K}_h(\cdot - \Catom{i})$, we obtain a
closed form for \emph{any} kernel \citep{wand_kernel_1994}. This is established in the following proposition.

\begin{proposition}[Closed-form smoothed CDF and density]
	\label{prop:closedform}
	For any kernel $K$ with CDF $\mathbb{K}$,
	\begin{equation}
		\label{eq:cdf-general}
		Q_n^{h}(y) = \sum_{i=1}^{M} \Delta_i\, \mathbb{K}\!\left(\frac{y-\Catom{i}}{h}\right),
		\qquad
		\hat{f}^{h}(y) = \frac{d}{dy}Q_n^{h}(y)
		= \sum_{i=1}^{M} \Delta_i\, \frac{1}{h} K\!\left(\frac{y-\Catom{i}}{h}\right).
	\end{equation}
	$Q_n^{h}$ is non-decreasing (non-negative weights, monotone $\mathbb{K}$) and
	bounded in $[0, \sum_i \Delta_i]$; $\hat{f}^h$ is a weighted kernel density
	estimate on the atoms and integrates to $\sum_i \Delta_i$.
\end{proposition}
The proof of this proposition is given in Appendix \ref{proof:closedform-smoothing}.
Proposition \ref{prop:closedform} has three fundamental implications:
\begin{itemize}
		\item for $\Delta_i \geq 0$, the associated smoothed CDF inherits its monotonicity from the kernel;
		\item the closed-form expressions in \eqref{eq:cdf-general} eliminate the necessity for numerical approximation of both the smoothing convolution integrals, and the finite difference approximations of its associated density;
		\item as long as $\sum_i \Delta_i = 1$, one naturally gets the correct limits for the smoothed CDF, and the normalization of its associated density.
\end{itemize}
In other words, kernel-smoothing a conformal distribution is exactly a weighted KDE placed on the conformal
atoms \citep{parzen_estimation_1962, rosenblatt_remarks_1956}.
\subsection*{The fidelity-constrained bandwidth problem}

Formulae \eqref{eq:cdf-general} provide closed-form expressions for any bandwidth $h$ for the smoothed conformal CDFs and their associated \emph{analytical} densities. However, they do not necessarily satisfy the perturbation condition in Theorem~\ref{thm:approximation_pit_guarantee}, i.e. there is no a-priori guarantee that the smoothed CDF remains $\varepsilon$-close to the (corrected) conformal CDF, which, in particular, implies that one cannot use the PIT-perturbation guarantee in \eqref{eq:approximation_pit_guarantee} to quantify validity loss. We address this by formulating an associated \emph{fidelity-constrained} optimization problem of the bandwidth, that imposes $\varepsilon$-distance from the conformal CDF, hence preserves the asymptotic coverage guarantee.

Before presenting the associated optimization problem, we first address what is the \textit{best} one can expect in terms of distance at conformal atoms. Formally, we specify a \emph{target} $t_j$ and constrain the smoothed CDFs to be $\varepsilon$-close to the conformal CDF at atom $j$, $d_j(h)=Q_n^h(C_{(j)}) - t_j\leq \varepsilon$. This problem is ill-posed when one naively sets the target to be the (right-continuous) value of the piecewise constant conformal CDF.

At an atom $j$ the step CDF is discontinuous, with different left and right limits $F_{j-1}$
and $F_j$. With a careful inspection of \eqref{eq:cdf-general}, we find
\begin{equation}
	\label{eq:atom-split}
	Q_n^{h}(\Catom{j})
	= \underbrace{\sum_{i<j}\Delta_i\,\mathbb{K}\!\left(\tfrac{\Catom{j}-\Catom{i}}{h}\right)}_{\to\,F_{j-1}\equiv \sum_{i<j} \Delta_i\ \text{as }h\to0^+}
	\;+\; \underbrace{\Delta_j\,\mathbb{K}(0)}_{=\,\frac12\Delta_j\ \text{(self term)}}
	\;+\; \underbrace{\sum_{i>j}\Delta_i\,\mathbb{K}\!\left(\tfrac{\Catom{j}-\Catom{i}}{h}\right)}_{\to\,0\ \text{as }h\to0^+}.
\end{equation}
This means that every conformal atom's own jump contributes its half-mass $\tfrac{1}{2}\Delta_j$ to the smoothed CDF, as for any (symmetric) kernel $\mathbb{K}(0)=1/2$. Since in the $h\to0^+$ limit all higher atoms' contributions vanish, this means that the midpoint $F_{j-1}+\tfrac{1}{2}\Delta_j$ is the limit of the smoothed CDF for vanishing bandwidths.

The discussion above motivates the introduction of the following fidelity target.
\begin{definition}[Fidelity target]
	\label{def:target}
	Given the jump mass in \eqref{eq:jump_mass}, the \emph{midpoint target} at atom $\Catom{i}$ is
	\begin{equation}
		t_i := \tfrac12\bigl(F_{i-1} + F_i\bigr) = F_{i-1} + \tfrac12 \Delta_i .
	\end{equation}
\end{definition}

The choice of $t_i$ in Definition~\ref{def:target} is not a modelling
convention but the unique choice attainable in the $h\to0^+$ limit. Nonetheless, using this fidelity target one preserves a perturbation PIT guarantee as stated in the following corollary, which is a direct consequence of Theorem~\ref{thm:approximation_pit_guarantee}.
\begin{corollary}[Validity guarantee with $\varepsilon$-fidelity for midpoint targets]
Given a bandwidth $h$ that satisfies $\varepsilon$-fidelity at each conformal atom, one preserves an $\epsilon\coloneqq \varepsilon + \max_{i} \tfrac{1}{2}\Delta_i$ perturbed validity guarantee for the smoothed CDF
	\begin{align}\label{eq:validity_guarantee_after_smoothing:general}
		\sup_{u\in [0, 1]} \Big|\mathbb{P}(Q_n^{h, CPD}(Y_{n+1})\leq u) - u\Big|\leq \epsilon + \frac{1}{n-m+1},
	\end{align}
	and for evenly-spaced quantile matching
    \begin{align}\label{eq:validity_guarantee_after_smoothing:qm}
		\sup_{u\in [0, 1]} \Big|\mathbb{P}(Q_n^{h, QM}(Y_{n+1})\leq u) - u\Big|\leq  \varepsilon + \frac{3}{2K} + \frac{1}{n-m+1}.
	\end{align}
\end{corollary}

We require the smoothed CDF to reproduce the conformal predictive distribution
at every atom up to tolerance $\eps>0$. Putting $\Kbar := 1 - \mathbb{K}$, the signed
per-atom deviation from the target in Definition~\ref{def:target} reads as
\begin{equation}
	\label{eq:dev}
	d_j(h) := Q_n^{h}(\Catom{j}) - t_j
	= \sum_{i>j} \Delta_i\, \Kbar\!\left(\frac{\Catom{i}-\Catom{j}}{h}\right)
	- \sum_{i<j} \Delta_i\, \Kbar\!\left(\frac{\Catom{j}-\Catom{i}}{h}\right).
\end{equation}
In the above, we used the facts that $F_{j-1}\equiv \sum_{i<j} \Delta_i$ and that for a symmetric kernel $\mathbb{K}(y) = 1-\mathbb{K}(-y)=\bar{\mathbb{K}}(-y)$.
The two sums represent, respectively, the right-neighbour mass the kernel smears
below $\Catom{j}$ and the left-neighbour mass smeared above it. As $h \to 0$
every tail term vanishes, so $d_j(h) \to 0$: minimal smoothing reproduces the
midpoint targets exactly. The fidelity constraint is then given by
$\max_j |d_j(h)| \le \eps$, and the associated bandwidth optimization problem is defined as follows
\begin{problem}[Maximum-bandwidth smoothing]
	\label{prob:maxbw}
	Find the largest global bandwidth whose smoothed CDF is $\eps$-faithful at all
	atoms:
	\begin{equation}
		\label{eq:maxbw}
		h^{\star} = \max\Bigl\{ h > 0 : \max_{1\le j\le M} |d_j(h)| \le \eps \Bigr\}.
	\end{equation}
\end{problem}

Since larger $h$ produces a lower-variation density, Problem~\ref{prob:maxbw}
returns the smoothest\footnote{Note: the "mapping" $\textnormal{bandwidth}\mapsto \textnormal{smoothness}$ is not necessarily monotone, i.e. there may be smaller bandwidths that produce more regular versions of the CDF. The above reasoning is heuristic, justified by the limit behaviors $\lim_{h\to 0/\infty}$.} admissible density in the crude "as much blur as
allowed" sense. We now establish that it is well posed.

\subsection*{Feasibility, non-emptiness and attainment}

Let $\mathcal{F}_j := \{h>0 : |d_j(h)|\le\eps\}$ and
$\mathcal{F} := \bigcap_{j=1}^M \mathcal{F}_j$ for the feasibility sets per atom and in total, respectively, and put $h^\star = \sup\mathcal{F}$ for the optimal bandwidth of the optimization Problem~\ref{prob:maxbw}, and define $g\coloneqq \max_{1\le j\le M} |d_j(h)|$.
The following proposition establishes that the feasibility set is non-empty and the optimum is attained.

\begin{proposition}[Non-emptiness and attainment]
\label{prop:nonempty}
For any sufficiently small $\eps>0$ the feasible set $\mathcal{F}$ is non-empty, so $h^\star>0$ is
well defined. Moreover $\mathcal{F}$ is bounded above and the supremum in
\eqref{eq:maxbw} is attained.
\end{proposition}
The proof is given in Appendix~\ref{proof:nonempty}.

\subsection*{Existence of closed-form optima}
A crucial observation that allows for the closed-form solution of the optimization problem is that at the optimum some fidelity constraint is necessarily active. We record this
first, since it is what turns the definition $h^\star=\sup\mathcal{F}$ into a scalar root equation that solvers below directly
target.

\begin{proposition}[Constraint is saturated at the optimum]
\label{prop:active}
\sloppy
The maximizer $h^\star$ of Problem~\ref{prob:maxbw} satisfies
$\max_{1\le j\le M}|d_j(h^\star)|=\eps$; that is, at least one atom's deviation
attains the tolerance.
\end{proposition}
The proof is given in Appendix~\ref{proof:active}.
Conceptually, maximizing the strictly increasing map $h\mapsto h$ over a set bounded above
places the maximizer at the upper edge of the feasible set; nothing in the
objective rewards an interior point, so the constraint must saturate. The argument uses only
$\sup\mathcal{F}$, attainment, and continuity; it is indifferent to the
\emph{shape} of $\mathcal{F}$, and in particular remains valid even when
$\mathcal{F}$ is not a single interval. What the proposition does \emph{not} tell us is
\emph{which} atom binds, nor its sign; identifying $j^\star$ is exactly the
per-atom, per-interval enumeration carried out for the Epanechnikov kernel in
Theorem~\ref{thm:epa-agg}. Formally, Proposition \ref{prop:active} leads to the following optimization problem equivalent to Problem~\ref{prob:maxbw}.
\begin{problem}[Equivalent formulation of Problem~\ref{prob:maxbw}]
	\label{prob:maxbw:equiv}
	Find the largest global bandwidth whose smoothed CDF is $\eps$-faithful at all
	atoms:
	\begin{equation}
		\label{eq:maxbw:equiv}
		h^{\star} = \max\Bigl\{ h > 0 : \max_{1\le j\le M} |d_j(h)| = \eps \Bigr\}.
	\end{equation}
\end{problem} The tractability of solving the resulting equation \eqref{eq:maxbw:equiv} depends entirely on the kernel.
In order to ease the presentation, in what follows we use the following assumption.
\begin{assumption}[Connected feasible set]\label{assumption:connected}
 	The feasible region $\mathcal{F}$ is connected, equivalently a single interval of the form $\mathcal{F}=(0, h^\star]$.
 \end{assumption}
  In particular, for the values of $h^\star$ found below to be the true optima, we require that the first time $\max_{1\leq j\leq M}|d_j(h)|$ hits $\varepsilon$ from the left, is also the right-most end of the feasible region. In case this assumption is violated, optima below are conservative underestimations of the highest bandwidth. A sufficient condition ensuring connectedness is that $g$ is a non-decreasing mapping of $h$. We emphasize that in case of the Epanechnikov kernel, this assumption is mere presentational convenience, and the global optimum can still be recovered even if the feasible set is disconnected -- see Remark~\ref{remark:connected}.

\subsection*{Piecewise cubic closed-form solution for the Epanechnikov kernel}

The Epanechnikov kernel and its CDF are
\begin{equation}
	K_h(u) = \frac{3}{4h}\Bigl(1 - \tfrac{u^2}{h^2}\Bigr)\mathbf{1}\{|u|\le h\},
	\qquad
	\mathbb{K}(t) =
	\begin{cases}
		0, & t<-1,\\
		\tfrac12 + \tfrac34 t - \tfrac14 t^3, & -1\le t\le 1,\\
		1, & t>1,
	\end{cases}
\end{equation}
so the kernel CDF is a cubic polynomial. The smoothed objects are exact,
compactly supported piecewise polynomials,
\begin{equation}
	\label{eq:epa}
	Q_n^{h}(y) = \sum_{i=1}^{M} \Delta_i\, \mathbb{K}\!\left(\frac{y-\Catom{i}}{h}\right),
	\qquad
	\hat{f}^{h}(y) = \sum_{i=1}^{M} \Delta_i\, \frac{3}{4h}
	\Bigl(1 - \bigl(\tfrac{y-\Catom{i}}{h}\bigr)^2\Bigr)^{\!+},
\end{equation}
and each atom $i$ influences only the interval $[\Catom{i}-h, \Catom{i}+h]$.

Compact support makes the constraint local and polynomial. We derive
\eqref{eq:epa-cubic} in full, since the structure of the coefficients is what
makes the optimum solvable in closed form. The Epanechnikov complementary CDF is,
on the relevant range,
\begin{equation}
\label{eq:epa-Kbar}
\Kbar(z) = \tfrac12 - \tfrac34 z + \tfrac14 z^3 \quad (0\le z\le 1),
\qquad \Kbar(z)=0 \quad (z\ge 1),
\end{equation}
a cubic polynomial \emph{with no quadratic term}. By \eqref{eq:epa-Kbar}, a neighbour $i$ contributes to
$d_j$ in \eqref{eq:dev} only when its argument is below $1$, i.e.\
$|\Catom{i}-\Catom{j}|<h$; all others give
$\Kbar=0$ exactly. This defines the active set
$S_j(h):=\{i\ne j : |\Catom{i}-\Catom{j}|<h\}$. Write $s_i:=\operatorname{sgn}(\Catom{i}-\Catom{j})$, so the two signed sums of \eqref{eq:dev}
combine into one, and set $z_i:=|\Catom{i}-\Catom{j}|/h$ and $u:=1/h$. Then
\begin{align}\label{eq:epa:d_j:intermediate}
d_j(h) = \sum_{i\in S_j(h)} s_i\,\Delta_i\,\Kbar(z_i)
= \sum_{i\in S_j(h)} s_i\,\Delta_i\Bigl(\tfrac12 - \tfrac34 z_i + \tfrac14 z_i^3\Bigr).
\end{align}
Because $s_i=\pm1$ we have $s_i^3=s_i$, which lets the sign be absorbed into the odd
powers, giving the following expressions
\begin{equation*}
s_i\,z_i = s_i\,|\Catom{i}-\Catom{j}|\,u = (\Catom{i}-\Catom{j})\,u,
\quad
s_i\,z_i^3 = s_i\,|\Catom{i}-\Catom{j}|^3\,u^3 = (\Catom{i}-\Catom{j})^3\,u^3 .
\end{equation*}
Substituting into \eqref{eq:epa:d_j:intermediate} and collecting powers of $u$ identifies the three bracketed sums as
the coefficients of a cubic in $u$ in the binding equation of Problem~\ref{prob:maxbw:equiv}
\begin{align}\label{eq:epa-cubic}
d_j(h)
\begin{aligned}[t]
    = \underbrace{\tfrac12\!\!\sum_{i\in S_j(h)}\!\! s_i\Delta_i}_{\coloneqq c_0(h)}
\;-\; \tfrac34\Bigl(\underbrace{\sum_{i\in S_j(h)}\!\Delta_i(\Catom{i}-\Catom{j})}_{\coloneqq c_1(h)}\Bigr)u
\;+\; \tfrac14\Bigl(\underbrace{\sum_{i\in S_j(h)}\!\Delta_i(\Catom{i}-\Catom{j})^3}_{\coloneqq c_3(h)}\Bigr)u^3\\=\pm \varepsilon.
\end{aligned}
\end{align}
On any interval where the active neighbor set $S_j(h)$ is fixed, the coefficients $c_0, c_1, c_3$ are constant. Additionally, the binding condition $d_j(h)=\pm \varepsilon$ has no quadratic term, and is therefore solvable by Cardano's formula. Since $S_j(h)$ changes only at the breakpoints $h=|C_{(i)}-C_{(j)}|$, $d_j(h)$ is piecewise cubic in $u\equiv 1/h$, and solving its associated cubic roots on each active neighbor set gives all possible bandwidths where one atom saturates the feasibility constraint. This observation leads to the following aggregation of a \emph{global optimum}.
\begin{theorem}[Aggregation to a global optimum]
\label{thm:epa-agg}
Suppose Assumption~\ref{assumption:connected} holds.
Then the feasible region connected to $h\to 0^+$ ends at the first bandwidth at which
\emph{any} atom's deviation reaches $\eps$; hence
\begin{equation}
h^{\star} = \min_{1\le j\le M} \inf\{ h>0 : |d_j(h)| = \eps \},
\end{equation}
i.e.\ the global minimum over all valid crossings obtained by solving the cubics
in \eqref{eq:epa-cubic}. Each per-atom, per-interval crossing is a closed-form
Cardano root; $h^\star$ is the minimum of the finitely many valid ones.
\end{theorem}
\begin{proof}
All $d_j(0^+)=0<\eps$, so for $h$ below the smallest crossing no constraint is
violated and the region is feasible; at the smallest crossing the first
constraint becomes active. Hence $h^\star$ equals that smallest crossing, which
is $\min_j$ of the per-atom first hits.
\end{proof}
The aggregation in Theorem~\ref{thm:epa-agg} relies on the connected feasible set in Assumption~\ref{assumption:connected}.
Nonetheless, in case of the Epanechnikov kernel, one can even recover the global optimum of Problem~\ref{prob:maxbw:equiv} when the feasible region is a union of disjoint intervals. 
\begin{remark}[Aggregation of disconnected feasible sets
]\label{remark:connected}
The proof of Theorem~\ref{thm:epa-agg} identifies $h^\star$ with the \emph{first} bandwidth at which any
constraint activates. This equals the global $\sup\mathcal{F}$ precisely when
$\mathcal{F}$ is the single interval $(0,h_{\mathrm{first}}]$. However, in case of the Epanechnikov kernel the global optimum can be obtained even if this assumption is violated.
Solving each piecewise cubic equation in \eqref{eq:epa-cubic} gives at most $2\times3\times M\times (M-1)$ roots that partition the positive real half line into intervals whose endpoints bind the feasibility condition for an atom. Since, as established in Proposition~\ref{prop:nonempty}, the feasibility region is bounded and the mapping $h\mapsto \max_{1\leq j\leq M}|d_j(h)|$ is continuous, it then suffices to evaluate the feasibility of these sub-intervals at the midpoint. The global optimum coincides with the right endpoint of the right-most feasible sub-interval.
\end{remark}
\section{Simulated data}
\par We now compare the three main approaches for retrieving predictive densities on a large simulated dataset: direct finite differencing, optimal kernel smoothing of the noisy CDF, and quantile matching with $K=100$ evenly spaced quantile levels. The dataset simulates house transaction prices with realistic features, resembling the data used in the Real Estate Valuation department of Ortec Finance. The simulation model is described below and is the same as in \cite{francke}. Throughout this section, the target random variable $Y$ represents transaction prices \textit{after} a $\log$ transformation.
\subsection*{The Hierarchical Trend Model}

The transaction prices are simulated using the Hierarchical Trend Model (HTM) \citep{htm-1, htm-2}:
\begin{align}\label{eq:htm}
\begin{split}
    \mathbf{y}_t &\sim \mathcal{N}\!\left(
\ind_{n_t}\mu_t
+ D_{\boldsymbol{\lambda},t}\boldsymbol{\lambda}_t
+ D_{\boldsymbol{\theta},t}\boldsymbol{\theta}_t
+ X_t\beta
+ D_{\boldsymbol{\eta},t}\boldsymbol{\eta}
+ \varepsilon_t,\;
\sigma_\varepsilon^2 I_{n_t}
\right), \\[6pt]
\Delta \mu_t &\sim \mathcal{N}
\left(\rho \Delta \mu_{t-1} + \alpha(1-\rho),\;
\sigma_\mu^2 \right),
\quad
\Delta \mu_1 \sim \mathcal{N}
\left(\alpha,\;
\frac{\sigma_\mu^2}{1-\rho^2}\right),
\quad
\mu_1 = 0, \\[6pt]
\boldsymbol{\lambda}_{t+1} &\sim \mathcal{N}
\left(\boldsymbol{\lambda}_t,\; \sigma_{\boldsymbol{\lambda}}^2 I_{n_t}\right),  \\[6pt]
\boldsymbol{\theta}_{t+1} &\sim \mathcal{N}
\left(\boldsymbol{\theta}_t,\; \sigma_{\boldsymbol{\theta}}^2 I_{n_t}\right), \\[6pt]
\boldsymbol{\eta} &\sim \mathcal{N}\left(0,\sigma_{\boldsymbol{\eta}}^2\right). 
\end{split}
\end{align}

Here, $\mathbf{y}_t$ is an $n_t \times 1$ vector of $\log$ prices, and $t$ indicates time in quarters of a year. The variable $\mu_t$ represents the log of a common price index, whose return $\Delta \mu_t$ follows an AR(1) model with lag coefficient $\rho$ and unconditional mean $\alpha$. The vectors $\boldsymbol{\lambda}_t$ and $\boldsymbol{\theta}_t$ represent log indexes, specified as random walks, for different regions and house types in deviation from $\mu_t$. Property characteristics (floor area, lot area, year of construction) are stored in $X$ with coefficients $\beta$. The matrices $D_{\boldsymbol{\lambda},t}$, $D_{\boldsymbol{\theta},t}$ and $D_{\boldsymbol{\eta},t}$ are selection matrices and $\boldsymbol{\eta}$ contains neighborhood random effects.

\par As in \cite{francke}, the posterior predictive distribution for each transaction price is approximated using 4000 samples, and a random one is used as the simulated transaction price $y_{n_t+1}$. The conditional distribution is normal given all features, so a true underlying distribution is known and serves as a benchmark. The dataset is split 65/25/10: 715{,}111 training points, 275{,}043 calibration points, and 110{,}018 test points. A LightGBM model \cite{LightGBM} is used for point and quantile prediction.

\subsection*{Scoring rules for forecast evaluation} Besides calibration, we assess the sharpness and overall quality of predictive distributions using proper scoring rules \cite{proper-scoring-rules}. Specifically, we consider the mean integrated square error (MISE), Continuous Ranked Probability Score (CRPS), the Quadratic Score (QS), and the Dawid--Sebastiani Score (DSS). The QS evaluates predictive densities by rewarding probability mass assigned to the realized outcome while penalizing excessive concentration. The CRPS compares the predictive distribution to the realized value and jointly measures calibration and sharpness. Since conformal predictive distributions are fuzzy, CRPS is computed using the crisp modification in~\eqref{crisp}; see \cite[Section~5]{split-cpd}. Finally, the DSS evaluates predictive distributions using only their first two moments, combining a standardized squared prediction error with a penalty that rewards sharper forecasts \cite{dawid-sebastiani}. For all three scoring rules, smaller values indicate better predictive performance.

In addition to the scoring rules summarized above, we compare the performance of the different approaches by considering both tails of the predictive distribution. Specifically, we compute the conditional means below the predicted 0.05 quantile and above the predicted 0.95 quantile.
Since the true underlying conditional distribution from the HTM model in \eqref{eq:htm} is normal, the corresponding true tail means can be computed analytically.

\begin{figure}
    \centering
    \includegraphics[width=0.7\linewidth]{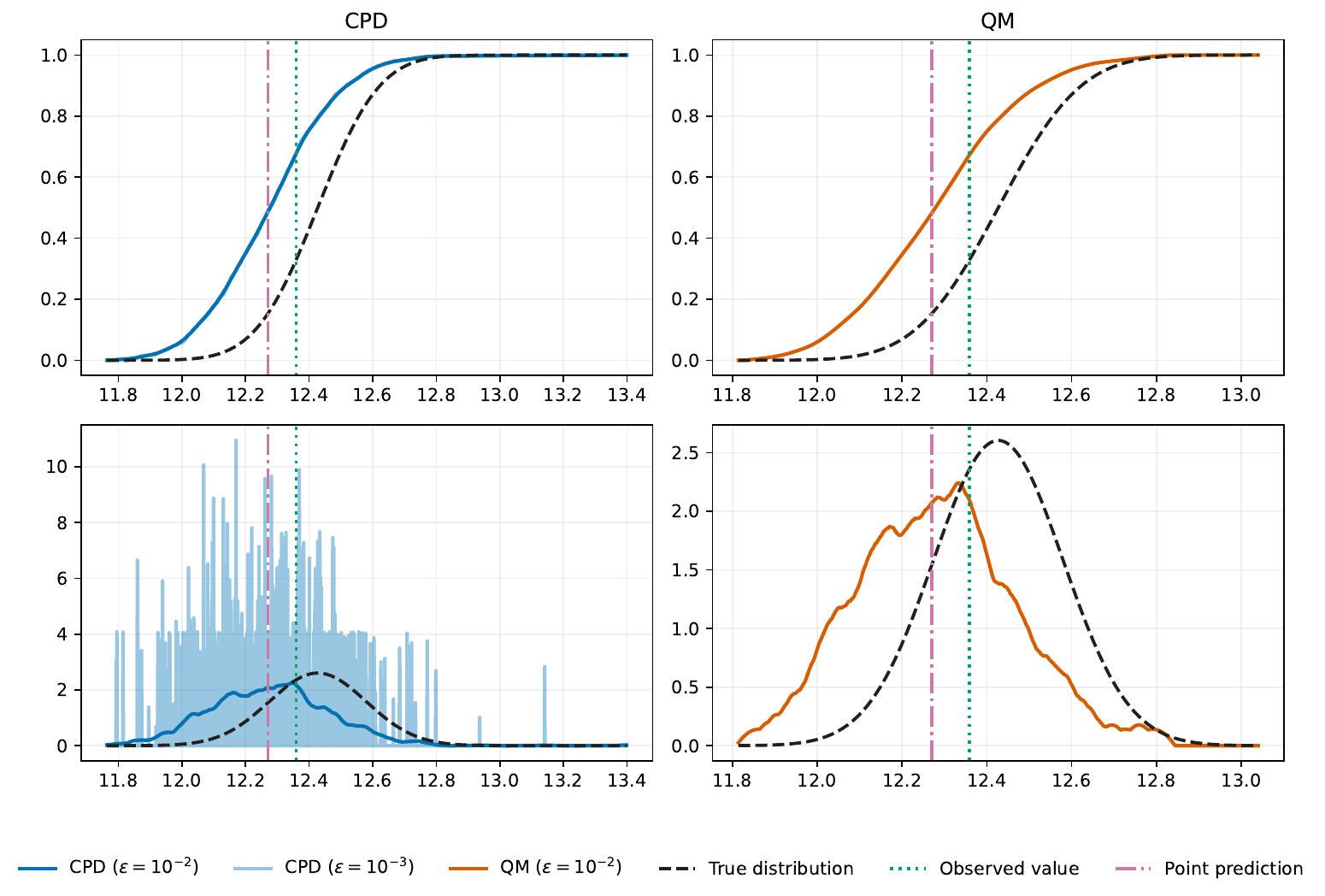}
    \caption{Predictive CDFs and densities for CPD and QM for a random transaction price. CPD is shown at $\epsilon = 10^{-3}$ and $\epsilon = 10^{-2}$
  to illustrate the effect of the smoothing tolerance. Here $|\mathcal{I}_2|=1000$.}
    \label{fig:SCP_one}
\end{figure}

\begin{figure}
    \centering
    \includegraphics[width=0.7\linewidth]{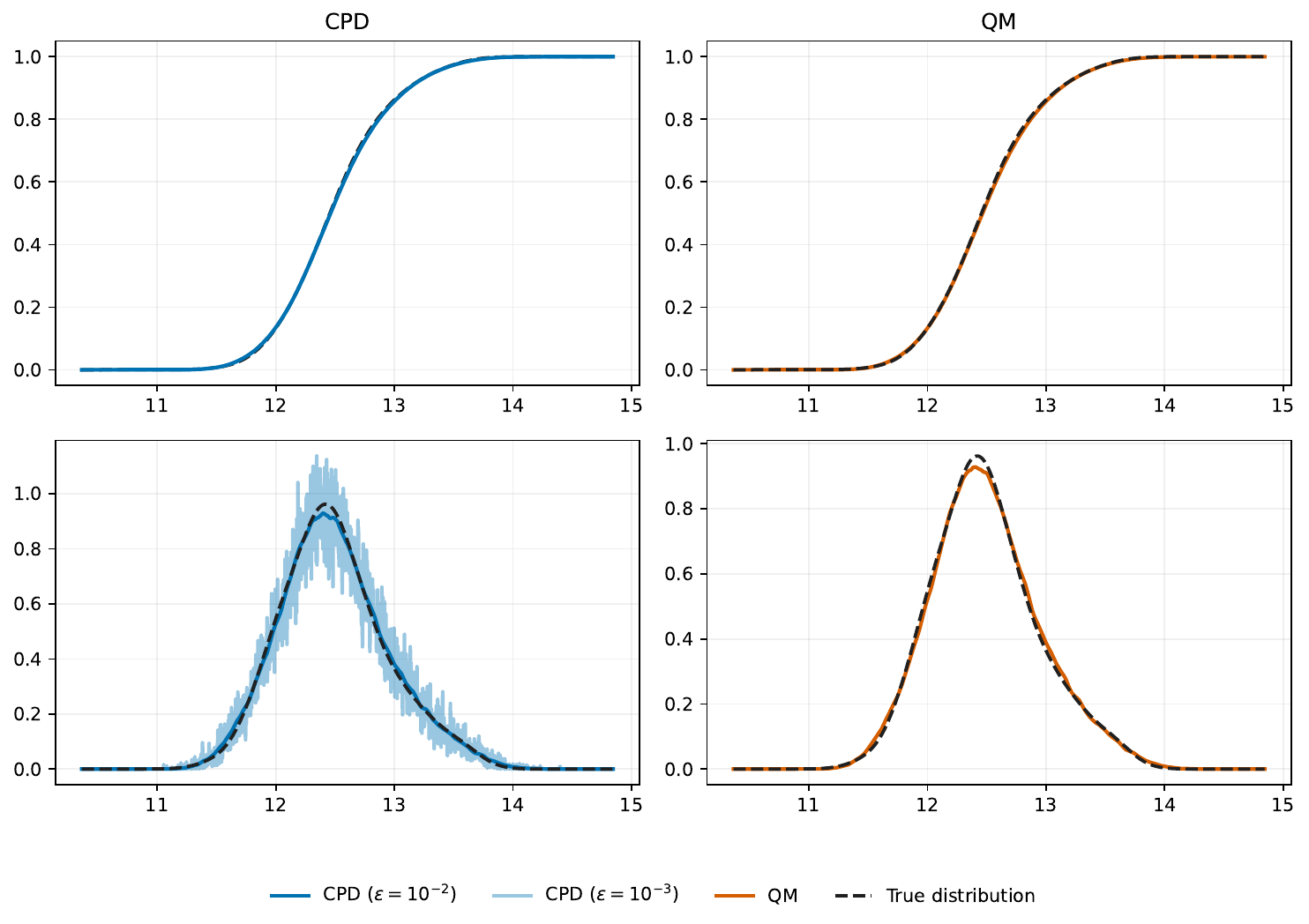}
    \caption{Marginal predictive CDFs and densities for CPD and QM over 200 samples. CPD is shown at $\epsilon = 10^{-3}$ and $\epsilon = 10^{-2}$
  to illustrate the effect of the smoothing tolerance. Here $|\mathcal{I}_2|=1000$.}
    \label{fig:SCP_marginal}
\end{figure}


\subsection*{Figures}

Figure \ref{fig:SCP_one} shows the conformal predictive distribution for a randomly selected transaction price, together with the true conditional distribution and the corresponding predictive density obtained by applying the Epanechnikov smoothing procedure described in Section~\ref{sec:optimal_filtering}. The calibration-set size is fixed at 1,000 because the bandwidth solver described in Section~\ref{sec:optimal_filtering} scales quadratically with the number of conformal atoms. Consequently, computing smoothed CPD densities becomes intractable for larger calibration sets. In such cases, numerical differentiation provides a more computationally feasible alternative; see, for example, Table \ref{tab:scoring_results}, where finite differencing is used. Similarly, Figure \ref{fig:SCP_marginal} presents the predictive CDFs and PDFs marginalized over 200 test observations. The Monte Carlo averages become close to the true marginal distribution, illustrating the ``on-average'' nature of the marginal validity guarantee. This empirical agreement should not, however, be interpreted as convergence of each estimated conditional distribution to its true counterpart. The figure also shows that the bandwidth selected for Epanechnikov smoothing is not always well suited to keeping the perturbation threshold $\epsilon$ on the same scale as the theoretical upper bound on the CPD error, which is inversely proportional to the calibration-set size; see Theorem \ref{distance-tailcorrected}.
\begin{figure}
    \centering
    \includegraphics[width=0.7\linewidth]{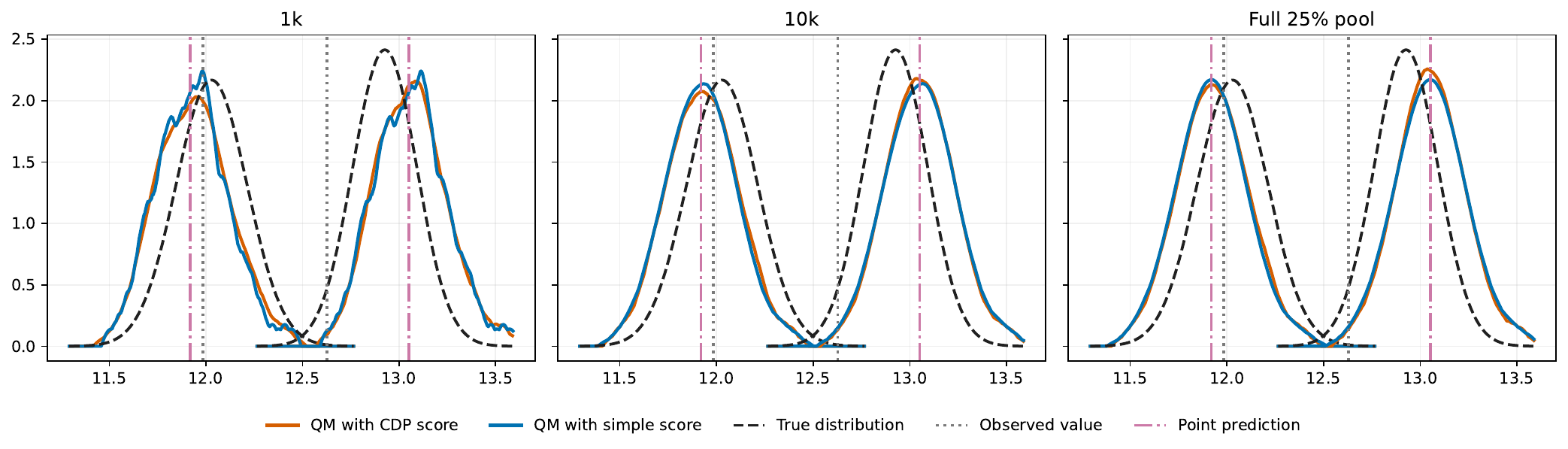}
    \caption{Predictive densities for two test observations using QM with the simple and the CDP conformity scores, compared with the true conditional distributions.}
    \label{fig:QM_CDP_marginal}
\end{figure}
\begin{figure}
    \centering
    \includegraphics[width=0.7\linewidth]{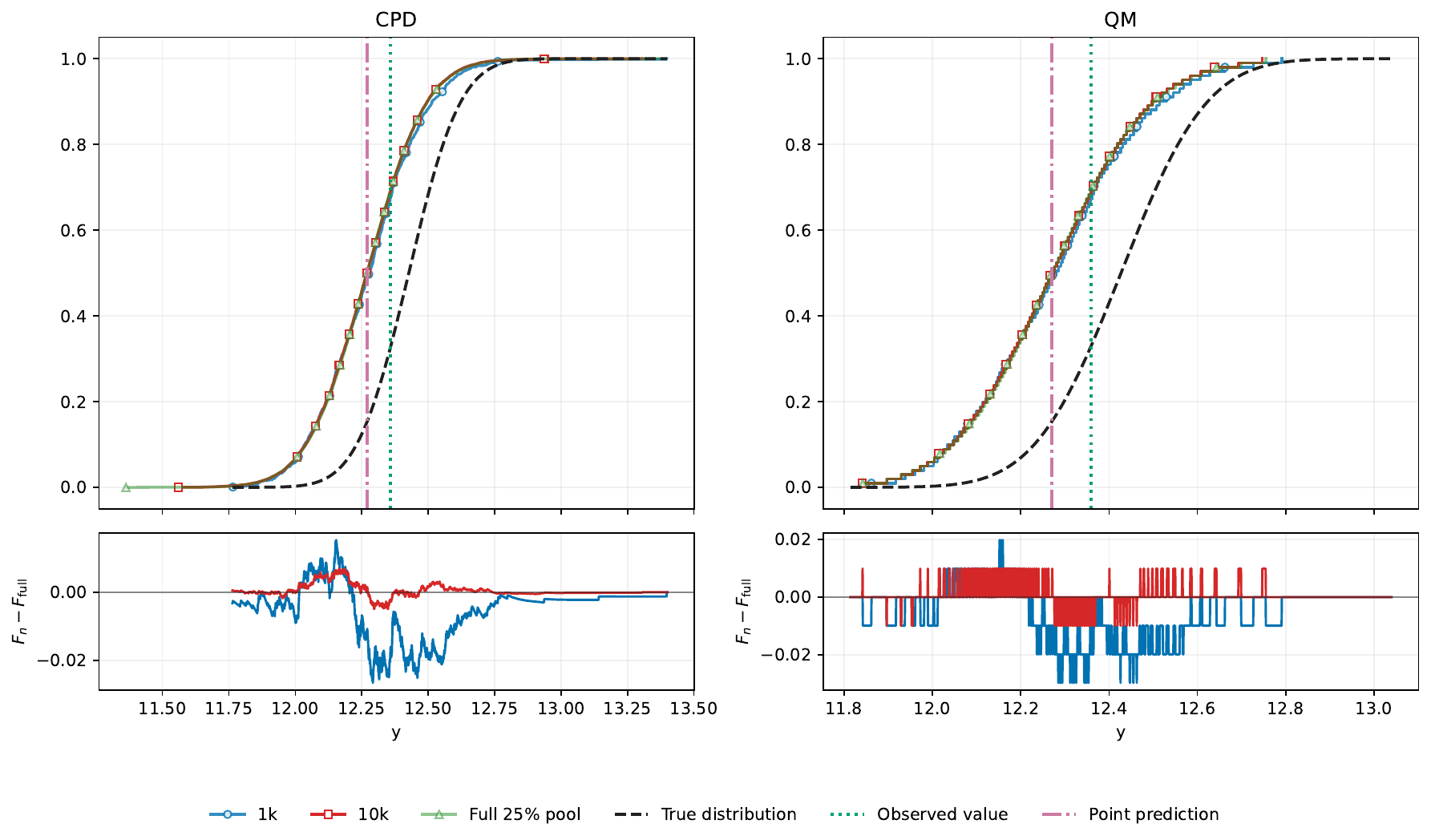}
    \caption{Step CDFs for CPD and QM across three calibration-set sizes. The lower plots show deviations from the CDF obtained with the 25\% calibration split, making overlap and convergence more visible.}
    \label{fig:cdf_calibration_comparison}
\end{figure}
\begin{figure}
    \centering
    \includegraphics[width=0.8\linewidth]{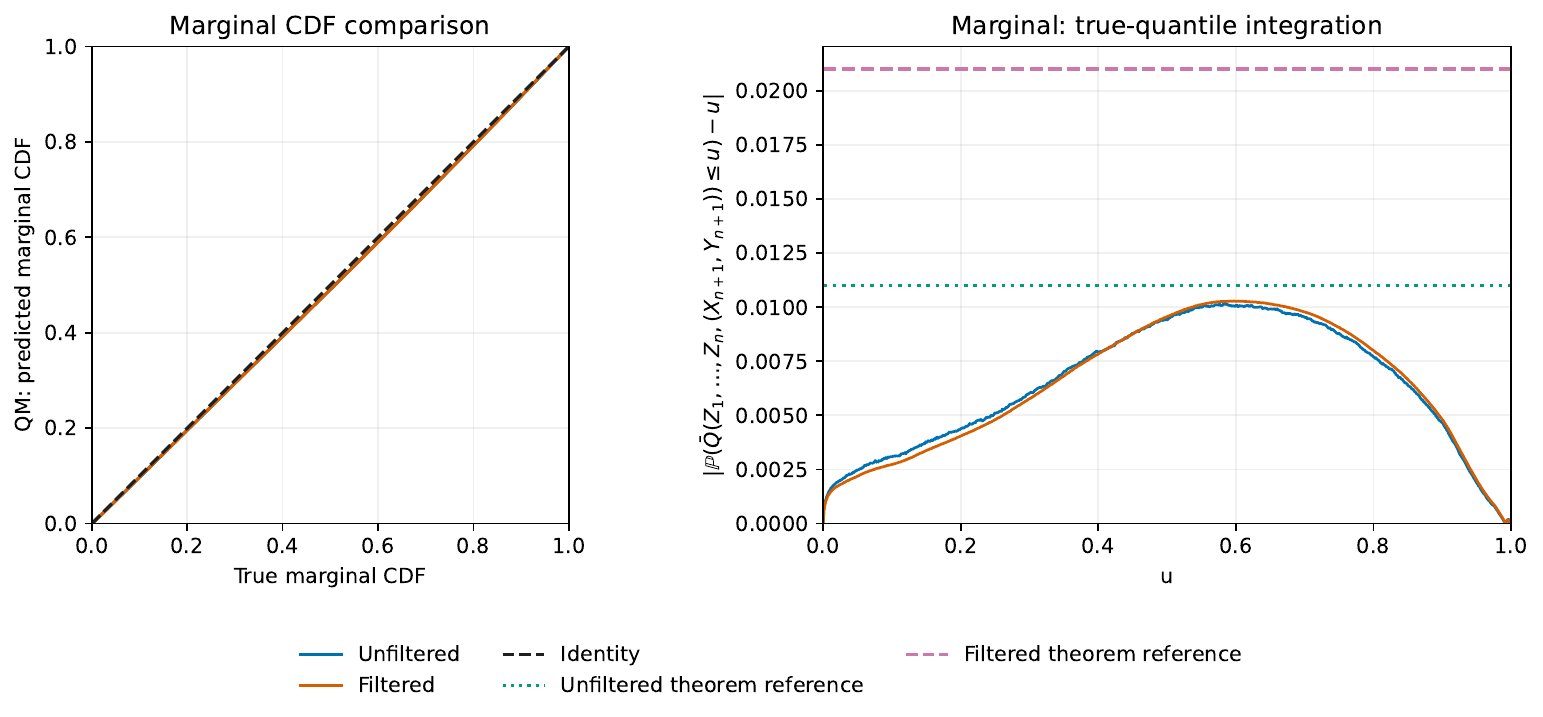}
    \caption{Diagnostics checks: On the left QQ-plot comparing the quantiles of the predicted marginal QM distribution and the true marginal distribution approximated on $10,000$ test samples. On the right an empirical validation of Theorem \ref{qm-theorem}.}
    \label{fig: probably correct}
\end{figure}
\par Figure~\ref{fig:QM_CDP_marginal} displays the results obtained when quantile matching is combined with the CDP conformity score for two representative observations. As the number of calibration observations increases, the estimated predictive distributions become progressively closer to the true conditional distributions, providing an empirical illustration consistent with Theorem~\ref{qm-theorem}. As discussed in Remark~\ref{remark:cdp}, no monotonicity violations were encountered in our experiments, allowing the smoothing procedure of Section~\ref{sec:optimal_filtering} to be applied to CDP-based quantile-matched distributions as well. A notable difference from the standard conformity score is that CDP produces genuinely heteroscedastic predictive distributions, whose shape and spread vary across observations. In contrast, predictive distributions obtained with the standard conformity score are essentially translated versions of the same residual distribution, differing primarily through the point prediction.
Figure \ref{fig:cdf_calibration_comparison} illustrates the effect of the calibration-set size on the predictive distribution for a single test observation. For the CPD method, increasing the calibration-set size reduces the jump sizes of the empirical predictive CDF and therefore produces a visibly smoother distribution.
Finally, Figure \ref{fig: probably correct} provides us diagnostics for the QM predictions, displaying that the predicted marginal CDF is aligned with the True marginal CDF, and that the threshold in Theorem \ref{qm-theorem} and its filtered counterpart, are satisfied.
\subsection*{Numerical results}
Table~\ref{tab:scoring_results} compares CPD and quantile matching (QM), using either finite differences or analytically optimized Epanechnikov smoothing, across the three calibration-set sizes. Epanechnikov smoothing is unavailable for CPD at the two larger sizes because its bandwidth optimization scales quadratically with the number of conformal atoms. The CRPS is essentially constant across all configurations, ranging only from $0.10737$ to $0.10849$, and is therefore uninformative for distinguishing the methods.
\par QM performs particularly well at the smallest calibration size, reducing the MISE from $0.4741$ for finite-difference CPD to approximately $0.18$. Unlike CPD, QM also remains compatible with analytically optimized Epanechnikov smoothing at every calibration size because it uses only $K=100$ atoms. This combination produces consistently strong QS and DSS values, while maintaining competitive MISE.

\par
QM also gives the smallest left- and right-tail mean errors across all calibration sizes. Finite-difference QM is especially efficient, requiring only $0.001$--$0.002$ seconds, and attains the lowest MISE at the full calibration size. Overall, QM provides accurate tail estimates, competitive density estimates, and a computationally tractable route to optimal Epanechnikov smoothing even for very large calibration sets.

\begin{table}[H]
\centering
\caption{Average scoring metrics over 200 test observations for three calibration set sizes: $|\mathcal{I}_2| = 1{,}000$, $|\mathcal{I}_2| = 10{,}000$, and $|\mathcal{I}_2| = 275{,}043$ (full 25\% split), using the SCP score. Quantile matching is applied using 100 evenly spaced quantile levels. Tail means are evaluated at the 0.05- and 0.95-quantile levels. Epanechnikov results for CPD at the two larger calibration sizes are unavailable because the bandwidth solver is computationally intractable for the corresponding number of atoms.}\label{tab:scoring_results}
\setlength{\tabcolsep}{5pt}
\resizebox{\textwidth}{!}{%
\begin{tabular}{llcccccc}
\toprule
& & \multicolumn{6}{c}{\textbf{Calibration size} $|\mathcal{I}_2| = 1{,}000$} \\
\cmidrule(lr){3-8}
\textbf{Method} & \textbf{Density} & \textbf{MISE} & \textbf{QS}
& \textbf{DSS} & \textbf{Left tail MAE} & \textbf{Right tail MAE}
& \textbf{Time (s)} \\
\midrule
CPD & Epanechnikov (opt.\ bw)
& $3.4397$ & $1.697$ & $-0.791$ & $0.0878$ & $0.1068$ & $2.235$ \\
CPD & Finite differences
& $0.4741$ & $-1.260$ & $-2.284$ & $0.0666$ & $0.0985$ & $0.004$ \\
\addlinespace
QM & Epanechnikov (opt.\ bw)
& $0.1793$ & $-1.512$ & $-2.323$ & $0.0608$ & $0.0791$ & $0.063$ \\
QM & Finite differences
& $0.1775$ & $-1.516$ & $-0.580$ & $0.0602$ & $0.0737$ & $0.002$ \\
\midrule
& & \multicolumn{6}{c}{\textbf{Calibration size} $|\mathcal{I}_2| = 10{,}000$} \\
\cmidrule(lr){3-8}
\textbf{Method} & \textbf{Density} & \textbf{MISE} & \textbf{QS}
& \textbf{DSS} & \textbf{Left tail MAE} & \textbf{Right tail MAE}
& \textbf{Time (s)} \\
\midrule
CPD & Epanechnikov (opt.\ bw)
& \multicolumn{6}{c}{N/A} \\
CPD & Finite differences
& $0.2227$ & $-1.433$ & $-2.272$ & $0.0760$ & $0.0703$ & $0.017$ \\
\addlinespace
QM & Epanechnikov (opt.\ bw)
& $0.1704$ & $-1.521$ & $-2.333$ & $0.0621$ & $0.0671$ & $0.066$ \\
QM & Finite differences
& $0.1691$ & $-1.524$ & $-2.076$ & $0.0618$ & $0.0647$ & $0.001$ \\
\midrule
& & \multicolumn{6}{c}{\textbf{Calibration size} $|\mathcal{I}_2| = 275{,}043$ (25\% split)} \\
\cmidrule(lr){3-8}
\textbf{Method} & \textbf{Density} & \textbf{MISE} & \textbf{QS}
& \textbf{DSS} & \textbf{Left tail MAE} & \textbf{Right tail MAE}
& \textbf{Time (s)} \\
\midrule
CPD & Epanechnikov (opt.\ bw)
& \multicolumn{6}{c}{N/A} \\
CPD & Finite differences
& $0.1686$ & $-1.533$ & $-2.271$ & $0.0760$ & $0.0703$ & $0.181$ \\
\addlinespace
QM & Epanechnikov (opt.\ bw)
& $0.1695$ & $-1.524$ & $-2.341$ & $0.0622$ & $0.0690$ & $0.318$ \\
QM & Finite differences
& $0.1673$ & $-1.527$ & $-1.836$ & $0.0614$ & $0.0652$ & $0.001$ \\
\bottomrule
\end{tabular}}
\end{table}

\subsection*{Acknowledgements}
This work was supported by Ortec Finance, which employed Dan Andrei Tudor during the master's thesis research that forms the basis of this study.
The authors would like to thank Afrasiab Kadhum and Marc Francke for providing the simulated dataset.

\appendix
\section{Proofs for the optimal kernel smoothing of conformal predictive distribution}
In what follows we give full details on the proofs relating the optimal bandwidth for kernel smoothing under fidelity constraints in Section~\ref{sec:optimal_filtering}.
\begin{proof}[Proof of Proposition~\ref{prop:closedform}]\label{proof:closedform-smoothing}
	Write the step CDF as the finite sum
	in \eqref{eq:heaviside_cdf}. The convolution $Q_n*K_h$ is
	well defined pointwise: $Q_n$ is bounded and $K_h\in L^1$, so
	$\int_{\R} |Q_n(y-u)|\,K_h(u)\,du \le \bigl(\sup|Q_n|\bigr)\int_{\R} K_h < \infty$.
	By linearity of convolution over the finite sum,
	\begin{equation*}
		Q_n^h = Q_n * K_h = \sum_{i=1}^M \Delta_i\,\bigl(H(\cdot-\Catom{i})*K_h\bigr).
	\end{equation*}
	We compute a single term directly. By the definition of convolution and of $H$,
	\begin{equation*}
		\bigl(H(\cdot-\Catom{i})*K_h\bigr)(y)
		= \int_{\R} H(z-\Catom{i})\,K_h(y-z)\,dz
		= \int_{\Catom{i}}^{\infty} K_h(y-z)\,dz .
	\end{equation*}
	Substitute $u = y-z$ (so $z=\Catom{i}\Rightarrow u=y-\Catom{i}$ and
	$z\to+\infty\Rightarrow u\to-\infty$, with $dz=-du$), then rescale $v=u/h$:
	\begin{align*}
		\bigl(H(\cdot-\Catom{i})*K_h\bigr)(y)
		\begin{aligned}[t]
		    &= \int_{-\infty}^{\,y-\Catom{i}} K_h(u)\,du
		= \int_{-\infty}^{\,y-\Catom{i}} \tfrac1h K\!\bigl(\tfrac{u}{h}\bigr)\,du\\
		&= \int_{-\infty}^{\,(y-\Catom{i})/h} K(v)\,dv
		= \mathbb{K}\!\left(\frac{y-\Catom{i}}{h}\right).
		\end{aligned}
	\end{align*}
	Summing over $i$ gives the stated CDF. Differentiating the finite sum term by
	term and using $\mathbb{K}'=K$ yields
	$\hat f^h(y) = \sum_i \Delta_i\,\tfrac1h K\bigl((y-\Catom{i})/h\bigr)$.
	
	For the stated properties: each $\mathbb{K}$ is non-decreasing (integral of
	$K\ge0$) and the weights $\Delta_i\ge0$, so $Q_n^h$ is non-decreasing;
	equivalently $\hat f^h\ge0$. Each $\mathbb{K}\in[0,1]$, so
	$Q_n^h\in[0,\sum_i\Delta_i]$, with $Q_n^h(y)\to0$ as $y\to-\infty$ and
	$Q_n^h(y)\to\sum_i\Delta_i$ as $y\to+\infty$. Finally, for each $i$,
	$\int_{\R}\tfrac1h K\bigl((y-\Catom{i})/h\bigr)\,dy = \int_{\R}K(v)\,dv = 1$, so
	$\int_{\R}\hat f^h = \sum_i\Delta_i$. For a proper CDF ($\sum_i\Delta_i=1$) the
	bounds specialise to $[0,1]$ and $\int_{\R}\hat f^h=1$.
\end{proof}

\begin{proof}[Proof of Proposition~\ref{prop:nonempty}]\label{proof:nonempty}
Throughout we take $Q_n$ to be a proper CDF, $\sum_i\Delta_i=1$ as produced by the tail corrections in Theorem~\ref{distance-tailcorrected}; the general case only rescales the bounds
below by $\sum_i\Delta_i$. Write $g(h) := \max_{1\le j\le M}|d_j(h)|$, so that
$\mathcal{F} = \{h>0 : g(h)\le\eps\}$ and $h^\star=\sup\mathcal{F}$. We argue in
four steps.

\smallskip
\noindent\emph{Step 1: $d_j(h)\to0$ as $h\to0^+$.} By \eqref{eq:dev}, $d_j$ is a
finite signed sum of terms $\Delta_i\,\Kbar(\delta_{ij}/h)$ with
$\delta_{ij}:=|\Catom{i}-\Catom{j}|>0$ fixed (atoms are distinct). As
$h\to0^+$, $\delta_{ij}/h\to+\infty$, and $\Kbar(x)=\int_x^\infty K\to0$ since
$K$ is integrable. Finitely many bounded terms each tend to $0$, so
$d_j(h)\to0$.

\smallskip
\noindent\emph{Step 2: a single uniform $h_0>0$ with $(0,h_0)\subseteq
\mathcal{F}$.} Fix $j$. By Step~1 there is $h_0^{(j)}>0$ with $|d_j(h)|\le\eps$
for all $h\in(0,h_0^{(j)})$. Set $h_0 := \min_{1\le j\le M} h_0^{(j)}$. This is a
minimum of \emph{finitely many} strictly positive numbers, hence $h_0>0$. For
any $h\in(0,h_0)$ and every $j$ we have $h<h_0\le h_0^{(j)}$, so
$|d_j(h)|\le\eps$; thus $g(h)\le\eps$ and $h\in\mathcal{F}$. Therefore
$(0,h_0)\subseteq\mathcal{F}$, so $\mathcal{F}\neq\emptyset$ and
$h^\star\ge h_0>0$. Finiteness of the atom set is essential: an infinite family
of thresholds $h_0^{(j)}$ could have infimum $0$, in which case no common
neighbourhood of the origin would be feasible.

\smallskip
\noindent\emph{Step 3: $\mathcal{F}$ is bounded above.} As $h\to\infty$,
$(\Catom{i}-\Catom{j})/h\to0$, so $\mathbb{K}\to\mathbb{K}(0)=\tfrac12$ and
$Q_n^h(\Catom{j})\to\tfrac12\sum_i\Delta_i=\tfrac12$; hence
$d_j(h)\to\tfrac12-t_j$. For the leftmost atom, $t_1=\tfrac12\Delta_1$, so
$d_1(h)\to\tfrac12-\tfrac12\Delta_1$, which exceeds $\eps$ for any reasonable
tolerance, certainly for $\eps<\tfrac12(1-\Delta_1)$. Thus $g(h)>\eps$ for all
large $h$ and $\mathcal{F}$ is bounded above. It suffices that \emph{some} $t_j$
lie outside $[\tfrac12-\eps,\tfrac12+\eps]$, which the extreme atoms guarantee.

\smallskip
\noindent\emph{Step 4: attainment.} Each $d_j$ is a continuous mapping of $h$ on $(0,\infty)$, as it is a
composition of the continuous $\mathbb{K}$ with fixed arguments
$(\Catom{j}-\Catom{i})/h$. Consequently, $g=\max_j|d_j|$ is continuous and
$\mathcal{F}=g^{-1}([0,\eps])$ is relatively closed in $(0,\infty)$. Since
$0<h_0\le h^\star<\infty$, pick $h_k\in\mathcal{F}$ with $h_k\to h^\star$;
continuity gives $g(h^\star)=\lim_k g(h_k)\le\eps$, so $h^\star\in\mathcal{F}$
and the supremum is attained.
\end{proof}
Notice that the proof constructs a conservative, not necessarily optimal, bandwidth explicitly, irrespective of the kernel used. Indeed, let $\ell_{\min}:=\min_k(\Catom{k+1}-\Catom{k})$ be the smallest inter-atom gap; every atom's nearest neighbour is adjacent, so $\delta_{ij}\ge \ell_{\min}$ for all $i\ne j$. Using the triangle inequality in \eqref{eq:dev}, monotonicity of $\Kbar$, and $\sum_i\Delta_i=1$,
	\begin{equation}
		\label{eq:h0-explicit}
		|d_j(h)| \;\le\; \sum_{i\ne j}\Delta_i\,\Kbar\!\Bigl(\tfrac{\delta_{ij}}{h}\Bigr)
		\;\le\; \Kbar\!\Bigl(\tfrac{\ell_{\min}}{h}\Bigr)\sum_{i\ne j}\Delta_i
		\;\le\; \Kbar\!\Bigl(\tfrac{\ell_{\min}}{h}\Bigr),
	\end{equation}
	uniformly in $j$. Hence $\Kbar(\ell_{\min}/h)\le\eps$ suffices for feasibility, i.e.\
	$h\le \ell_{\min}/\Kbar^{-1}(\eps)=:h_{c}$, well defined for any $\eps\in(0,\tfrac12)$
	since $\Kbar:[0,\infty)\to(0,\tfrac12]$ is continuous and strictly decreasing.

\begin{proof}[Proof of Proposition~\ref{prop:active}]\label{proof:active}
Write $g(h)=\max_j|d_j(h)|$. By Proposition~\ref{prop:nonempty}, $g$ is
continuous, $h^\star=\sup\mathcal{F}$ is finite, and $h^\star\in\mathcal{F}$
(attainment) with $h^\star\ge h_0>0$, so $h^\star$ lies in the interior of the
domain $(0,\infty)$. Being in $\mathcal{F}$ gives $g(h^\star)\le\eps$. Suppose the
inequality were strict, $g(h^\star)<\eps$. By continuity the set
$g^{-1}\bigl((-\infty,\eps)\bigr)$ is open and contains $h^\star$, hence contains
an interval $(h^\star-\eta,\,h^\star+\eta)\subset(0,\infty)$ for some $\eta>0$.
Every point of that interval lies in $\mathcal{F}$; in particular some
$h'\in(h^\star,h^\star+\eta)$ is feasible and strictly exceeds
$\sup\mathcal{F}$, which is a contradiction. Therefore $g(h^\star)=\eps$.
\end{proof}


\vskip 0.2in
\bibliography{references.bib}
\bibliographystyle{plainnat}  
\end{document}